\documentclass[11pt,reqno]{amsart}
\usepackage{amsfonts,amssymb,amsmath,amsthm}
\usepackage[margin=1in]{geometry}
\usepackage[hidelinks]{hyperref}
\usepackage{cleveref}
\usepackage{enumitem}
\usepackage{microtype}
\usepackage{mathtools}
\usepackage{color}
\usepackage{tikz-cd}
\usepackage{comment}

\def\p#1{\left( #1 \right)}
\def\Z{\mathbb{Z}}
\def\Q{\mathbb{Q}}
\def\F{\mathbb{F}}
\def\Gal{\operatorname{Gal}}
\def\GL{\operatorname{GL}}

\def\det{\operatorname{det}}

\def\Frob{\operatorname{Frob}}
\def\Aut{\operatorname{Aut}}
\def\rad{\operatorname{rad}}

\def\kronecker#1#2{\p{\frac{#1}{#2}}}

\def\Z{\mathbb Z}

\def\Q{\mathbb Q}
\def\R{\mathbb R}
\def\C{\mathbb C}
\def\F{\mathbb F}

\def\fp{\mathfrak p}

\def\Z{\mathbb Z}

\def\Q{\mathbb Q}
\def\R{\mathbb R}
\def\C{\mathbb C}
\def\F{\mathbb F}

\def\fp{\mathfrak p}

\def\det{\operatorname{det}}

\def\End{\operatorname{End}}

\def\Hom{\operatorname{Hom}}

\def\Aut{\operatorname{Aut}}
\def\Gal{\operatorname{Gal}}

\def\Frob{\operatorname{Frob}}

\def\car{\operatorname{char}}
\def\exp{\operatorname{exp}}

\def\GL{\operatorname{GL}}
\def\GSp{\operatorname{GSp}}

\def\O{\operatorname{O}}

\def\log{\operatorname{log}}

\def\rad{\operatorname{rad}}

\def\zar{\operatorname{zar}}
\def\split{\operatorname{split}}

\def\p#1{\left( #1 \right)}
\def\kronecker#1#2{\p{\frac{#1}{#2}}}
\def\p#1{\left( #1 \right)}
\def\Z{\mathbb{Z}}
\def\Q{\mathbb{Q}}
\def\F{\mathbb{F}}

\def\conn{\operatorname{conn}}
\def\O{\operatorname{O}}
\def\Gal{\operatorname{Gal}}
\def\GL{\operatorname{GL}}

\def\GSp{\operatorname{GSp}}
\def\car{\operatorname{Char}}
\def\det{\operatorname{det}}

\def\Li{\operatorname{Li}}
\def\Frob{\operatorname{Frob}}
\def\Aut{\operatorname{Aut}}

\def\rad{\operatorname{rad}}

\newcommand{\ol}[1]{\overline{#1}}
\newcommand{\cal}[1]{\mathcal{#1}}

\theoremstyle{plain}
\newtheorem{theorem}{Theorem}
\newtheorem{corollary}[theorem]{Corollary}
\newtheorem{lemma}[theorem]{Lemma}
\newtheorem{proposition}[theorem]{Proposition}

\theoremstyle{definition}

\newtheorem{remark}{Remark}

\title{Distribution of supersingular primes for abelian surfaces}
\author[Wang]{Tian Wang}
\address{Department of Mathematics \&
Statistic, Concordia University, Montreal, Quebec, Canada}
\email{tian.wang@concordia.ca}
\subjclass[2010]{Primary 11G10, 11N45; Secondary 11G18, 11N36.}

\begin{document}

\begin{abstract}
Let $A/K$ be an absolutely simple abelian surface defined over a number field $K$. We give unconditional upper bounds for the number of prime ideals  $\mathfrak{p}$ of $K$ with norm up to $x$ such that $A$ has supersingular reduction at $\fp$. These bounds are obtained in three distinct settings, depending on the endomorphism algebra of $A$, namely, the case of trivial endomorphisms, real multiplication (RM), and quaternion multiplication (QM).   In the RM case and when $K=\Q$, our results further implies an  unconditional upper bound on the distribution of Frobenius traces of $A$.  Furthermore, in the RM setting, we study the distribution of the middle coefficients of Frobenius polynomials of $A$ at primes where the reduction of  $A$  splits.  
\end{abstract}

\maketitle

\section{Introduction}
Let $E/\Q$ be a non-CM elliptic curve. For a prime $p$ that does not divide the conductor $N_E$ of $E$, we denote by $\overline{E}_p$ the reduction of $E$ at $p$.  We say that $\overline{E}_p$ is  \emph{supersingular} if the $\ol{\F}_p$-points of the $p$-torsion group  $\overline{E}_p[p]$ of $\overline{E}_p$ is trivial and such a prime $p$ is  called a \emph{supersingular prime} of $E$. A famous open problem related to the distribution of supersingular primes for elliptic curves 
is the Lang-Trotter Conjecture \cite[p. 36]{LaTr1976}, which  implies that  for all sufficiently large real number $x$, 
\begin{equation}\label{L-T-elliptic}
\pi_{E, ss}(x) := \#\{p\leq x: p\nmid N_E, \overline{E}_p \text{ is supersingular}\} \sim C_E\frac{x^{\frac{1}{2}}}{\log x}
\end{equation}
for some constant $C_E\neq 0$ that depends only on $E$.  

The first nontrivial result related to this conjecture is due to Serre  \cite[I-25]{Se1968},   who showed that the set of supersingular primes of a non-CM elliptic curve has density zero \cite[I–25]{Se1968}. 
Later in \cite[Th\'{e}or\`{e}me 12, p. 357]{Se1981}, he obtained  the unconditional upper bound  $\pi_{E, ss}(x)\ll x(\log x)^{-\frac{3}{2}+\epsilon}$ for any  $\epsilon>0$, using the  effective Chebotarev Density Theorem \cite{LaMoOd1979} together with  properties of  $\ell$-adic Lie groups. Incorporating  a sieve theoretical lemma, the upper bound was subsequently improved by Wan \cite{Wa1990}  to
$\pi_{E, ss}(x)\ll x(\log x)^{-2+\epsilon}$ for any $\epsilon>0$. In 1986,  Elkies \cite{El1987} made a significant progress in this direction by showing that there are infinitely many supersingular primes for elliptic curves over $\Q$. 
Based on this work, Murty and Elkies  \cite{El1991} also proved the unconditional upper bound $\pi_{E, ss}(x)\ll x^{\frac{3}{4}}$. It is worth mentioning that this bound matches with Serre's result under the Generalized Riemann Hypothesis (GRH) for Dedekind zeta functions \cite[p. 323]{Se1981}. 

In contrast, if  $E/\Q$ 
has complex multiplication, the distribution of supersingular primes is well understood. In this case, Deuring’s criterion \cite{De1941} gives $\pi_{E, ss}(x)\sim x(2\log x)^{-1}$.

One of the main goals of this paper is to give  unconditional upper bounds for the number of supersingular primes of absolutely simple abelian surfaces. We begin by recalling some basic definitions and setting up the notation.
 Let $q$ be a prime power and  $B$ an abelian variety over a finite field $\F_q$. We say $B$ is a \emph{supersingular} abelian variety if it  is isogenous over $\ol{\F}_q$ to a product of supersingular elliptic curves (see \cite[Theorem (4.2)]{Oo1974}). Now  let  $A$ be an absolutely simple abelian surface defined over a number field $K$. For a prime $\fp$ that does not divide the conductor $N_A$ of $A$, we denote by  $\ol{A}_\fp$  the reduction of $A$ at $\fp$. This is an abelian variety defined over the  residual field $\F_\fp$.  
We call $\fp$  a {\emph{supersingular prime}} of $A$ if $\ol{A}_\fp$ is supersingular. 

 It is well-known that for an abelian surface $A/K$, the density of supersingular primes is zero upon extending the base field (see, e.g., \cite[Proposition 5.1, p. 61]{BaGo1997} or \cite[Corollary 2.9, p. 372]{Og1982}). However, much less is known about  the upper bounds of the associated counting function:
\begin{equation}\label{ss-surface}
\pi_{A, ss}(x) := \#\{\fp\in \Sigma_K: N(\fp)\leq x, \fp\nmid N_A, \ol{A}_\fp \text{ is supersingular}\},
\end{equation}
Even the conjectural asymptotics of $\pi_{A, ss}(x)$ are only known for special abelian surfaces, which we now introduce.

In analogy with the Lang–Trotter Conjecture, Bayer and Gonz\'{a}lez \cite[Conjecture 8.2]{BaGo1997} predicted the distribution of $p$-ranks of modular abelian varieties over $\Q$ that implies the asymptotic behavior of $\pi_{A, ss}(x)$ for modular abelian surfaces. 
More precisely, let $f=\sum_{n\geq 1} a_n q^n$,  where $q=\exp(2\pi i z)$, be a non-CM newform of weight 2, level $N$, and Nebentypus character $\epsilon$. Let $A_f/\Q$ be the abelian variety attached to $f$ via the Eichler–Shimura theory. Let $K_f=\Q\left(\{a_n\}_{(n, N)=1}\right)$ be the coefficient field of $f$ and let $F_f=\Q\left(\{a_n^2/\epsilon(n)\}_{(n, N)=1}\right)$. We assume $A_f$ is absolutely simple. \footnote{If $A_f$ is not absolutely simple, then it is necessarily isogenous to a power of some simple abelian  variety.}  If
\begin{equation}\label{GL2-type}
[K_f:\Q]=2 \text{ and } [F_f:\Q]=1,
\end{equation}
 then the conjecture predicts that 
$\pi_{A_f, ss}(x)\sim C_{f} x^{\frac{1}{2}}(\log x)^{-1}$; 
 if
\begin{equation}\label{GL2-type-2}
[K_f:\Q] =[F_f:\Q]=2,
\end{equation}
then it is expected that $\pi_{A_f, ss}(x)\sim C_{f} \log \log x$, where $C_{f}$ is an absolute  constant depending only on $f$. \footnote{Too see this, we also need to  use   \Cref{def-ss-surface} (3) to relate the $p$-rank of $\ol{A}_p$ and the supersingular properties of $\ol{A}_p$.} 

Though no general asymptotic formula is known for  $\pi_{A, ss}(x)$,  results concerning the distribution of Frobenius traces of $A$ provide useful upper bounds. In the case $K=\Q$, a conjecture of Cojocaru, Davis, Silverberg, and Stange \cite[p. 3562]{CoDaSiSt2017} predicts that there exists a constant  $C(A)$ depending only on $A$ such that  $\pi_{A, ss}(x)\ll C(A)x^{\frac{1}{2}}(\log x)^{-1}$ holds  \footnote{To derive this upper bound, we also use the fact that if $p\geq 17$,  then supersingularity of the reduction $\ol{A}_p$  implies   $a_p(A)=0$  (see \Cref{def-ss-surface}). }  for a family of abelian surfaces that  satisfy   the equidistribution assumption and that the image of the adelic Galois representation  of $A$ is open in 
$\GSp_{4}(\widehat{\Z})$.  Under certain assumptions, some  upper bounds have been established.
For example,  under GRH, 
we have \cite[Theorem 1]{CoWa2022}
$
\pi_{A, ss}(x) \ll x^{\frac{10}{11}}(\log x)^{-\frac{9}{11}}
$; assuming both GRH and Artin's Holomorphy Conjecture (AHC), we have \cite[Theorem 20, p. 629]{Be2016}
$
\pi_{A, ss}(x) \ll x^{\frac{9}{10}}(\log x)^{-\frac{3}{5}}
$; assuming  GRH, AHC, and a Pair Correlation Conjecture (PCC), we have \cite[Theorem 2]{CoWa2022}  $
\pi_{A, ss}(x) \ll x^{\frac{2}{3}}(\log x)^{\frac{1}{3}}
$. These results, while conditional, show that the set of supersingular primes of such abelian surfaces has density 0 over $\Q$. 

In this paper, we provide unconditional upper bounds for the counting function  
(\ref{ss-surface})
for various abelian surfaces $A/K$, classified by their endomorphism algebras. Similar to the case of elliptic curves,  the asymptotic behavior of  $\pi_{A, ss}(x)$ is expected to depend on the structure of the endomorphism ring of $A$.   As a preliminary step, we recall that, according to Albert’s classification of division algebras with positive involutions, the endomorphism algebra  $D := \End_{\ol{K}}(A)\otimes_{\Z} \Q$ of $A$ must be one of 
\begin{enumerate}
    \item $\Q$;
    \item a real quadratic field, in which case we say $A$ has \emph{real multiplication (by $D$)};
    \item an indefinite quaternion algebra over $\Q$,  in which case we say $A$ has \emph{quaternion multiplication (by $D$)};
    \item a CM quartic field, in which case we say $A$ has \emph{complex multiplication (by $D$)}.
\end{enumerate}

Our main result is the following.
\begin{theorem}\label{main-thm-1}
Let $A/K$ be an absolutely simple abelian surface. For all sufficiently large $x$, we have
\begin{enumerate}
    \item if $D=\Q$, then
     \[
   \pi_{A, ss}(x) \ll  
\frac{x (\log \log x)^{\frac{3}{2}}}{(\log x)^{\frac{3}{2}}};
   \] \label{ss-generic}
   \item if $\End_{\ol{K}}(A)=\End_{K}(A)$ and 
   $A$ has real multiplication, 
   then   
     \[
  \pi_{A, ss}(x) \ll \frac{x (\log \log x)^{2}}{(\log x)^{{2}}};
   \] \label{ss-RM}
      \item if $\End_{\ol{K}}(A)=\End_{K}(A)$  and 
      $A$ has quaternion multiplication,  then  
     \[
  \pi_{A, ss}(x) \ll \frac{x (\log \log x)^{2}}{(\log x)^{{2}}}.
   \] \label{ss-QM}
\end{enumerate}
Moreover, the implicit constant in each $\ll$ depends  on  $A$ and $K$.
\end{theorem}

\begin{remark}
In Case (1), our bound improves  the unconditional estimate given in \cite[Theorem 1]{CoDaSiSt2017}, which gives 
$\pi_{A, ss}(x) \ll_\epsilon \frac{x}{(\log x)^{\frac{9}{8}-\epsilon}}$ for any $\epsilon>0$. 

In Case (2)  and if $K=\Q$, the abelian surface $A$ is associated to a  modular form via  Eichler-Shimura theory. This is because  $A$ is a $\GL_2$-type abelian variety over $\Q$,  hence is modular \cite{Ri1992}.   
Moreover,  the counting function $\pi_{A, ss}(x)$ also informs the distribution of Frobenius traces of $A$ and is relevant to Lang–Trotter-type conjecture for abelian surfaces (see \Cref{rm-final}). 

Case (3) of \Cref{main-thm-1} does not occur if $K=\Q$  \cite[p. 618, Proposition 1.3]{DiRo2005}.

These findings unconditionally establish that the set of supersingular primes of an absolutely simple abelian surface has density 0 over the base field. 
\end{remark}

\begin{remark}\label{Th-Za}
The upper bounds in \Cref{main-thm-1} (2) and (3) bear similarity to the unconditional result of Thorner and Zaman \cite{ThZa2018} for the Lang–Trotter conjecture in the case of non-CM elliptic curves, namely,  
\[
\#\{p\leq x: p \nmid N_E,  a_p(E)=t\}\ll \frac{x(\log\log x)^2}{(\log x)^2},
\]
where $E/\Q$ is a non-CM elliptic curve, $N_E$ is the conductor of $E$, $a_p(E):=p+1-|\ol{E}_p(\F_p)|$ is the Frobenius trace of $E$ at $p$, and $t$ is an integer.  

However, extending such bounds to higher-dimensional abelian varieties introduces new difficulties. In addition to handling more cases of Galois representations and incorporating suitable sieve techniques, a key part of the proof is to establish  results such as \Cref{conj-prop-QM}, \Cref{gsp-conj-prop},  \Cref{conj-prop}, and \Cref{split-lem}, which are relatively straightforward for elliptic curves but more subtle in higher dimensions. These results are essential for applying the effective Chebotarev Density Theorem (\Cref{func-ECDT}) and the sieve method.
\end{remark}

We now outline the proof of  \Cref{main-thm-1}. First, we split the counting function $\pi_{A, ss}(x)$ into 
\begin{equation}\label{ss-split-counting}
\#\{\fp\in \Sigma_K: N(\fp)\leq x, \fp\nmid N_A, \ol{A}_\fp \text{  splits and is supersingular}\}
\end{equation}
and
\begin{equation}\label{ss-simple-counting}
\#\{\fp\in \Sigma_K: N(\fp)\leq x, \fp\nmid N_A, \ol{A}_\fp \text{ is simple and  supersingular}\}
\end{equation}
and give an upper bound for each.
Here, $\ol{A}_\fp$ splits means  it is isogenous over $\mathbb{F}_\fp$ to a product of two elliptic curves defined  over $\F_\fp$. 
Then, we use \Cref{ss-Weil-poly} 
to reduce (\ref{ss-split-counting}) and (\ref{ss-simple-counting}) into prime distribution problems  for the characteristic polynomials of Frobenius  of $\overline{A}_\fp$ (also called  the Frobenius polynomials of $A$ at $\fp$) satisfying certain algebraic conditions (see  \Cref{ss-Weil-poly}). For the analytic inputs, we first adapt the inclusion-exclusion argument by Wan \cite[Lemma 4.1, p. 263]{Wa1990} and use the the Brun–Titchmarsh theorem to obtain a useful lemma (see  \Cref{uncond-lem}). This enables us to give an upper bound by sums over auxiliary  primes $\ell$ in short intervals for new counting functions  (see (\ref{S-i-max-trivial}), (\ref{S-i-max-RM}), and (\ref{S-i-max-QM})). 
 We then apply effective versions of the Chebotarev Density Theorem (see  \Cref{ECDT}) to obtain optimal upper bounds for the counting functions under consideration. This requires careful selection of number fields and conjugacy classes.  To perform the explicit computations, we also need to study the group-theoretic properties for the images of the Galois representations of $A$. Related results are given in \Cref{S2-Galois-rep} and    \Cref{sec: matrix}.   

We now turn to a related counting problem concerning the distribution of the middle coefficients of the Frobenius polynomials $P_{A,\fp}(X)$, under the assumption that $A$ admits RM and  $\overline{A}_\fp$ splits.

Assume that the endomorphism algebra of $A$, denoted $D = \End_{\overline{K}}(A) \otimes_{\mathbb{Z}} \mathbb{Q}$, is a real quadratic field. For a prime $\fp \nmid N_A$, the reduction $\overline{A}_\fp$ is defined over the finite field $\mathbb{F}_\fp = \mathbb{F}_{p^k}$ for some integer $k \geq 1$. The Frobenius polynomial of $A$ at $\fp$ is then of the form
\[
P_{A, \fp}(X)=X^4+a_{1, \fp}X^3+a_{2, \fp}X^2+p^ka_{1, \fp}X+p^{2k}.
\]
We consider
\begin{equation}\label{thm-2-function}
\pi_{A, \split, g(\cdot)}(x) := \#\{\fp\in \Sigma_K: N(\fp)=p^k\leq x,  \fp \nmid N_A,  \ol{A}_\fp \text{ splits},  a_{2, \fp}=g(p)\},
\end{equation}
where $g(\cdot): \mathbb{N} \to \R$ is any  function on the set  of all natural numbers $\mathbb{N}$, satisfying
$|g(p)|\leq 6p$ for all rational primes  $p$. The subsequent result establishes unconditional upper bounds for  $\pi_{A, \split, g(\cdot)}(x)$.

\begin{theorem}\label{a2p-split}
  Let $A/K$ be an absolutely simple abelian surface such that $\End_{K}(A)= \End_{\ol{K}}(A)$ and assume
  $A$  has real multiplication by a real quadratic field $F$. Let $g(\cdot): \mathbb{N} \to \R$ be an arithmetic function  that satisfies
$|g(p)|\leq 6p$ for all rational primes $p$. The following statements hold. 
  \begin{enumerate}
      \item\label{(1)} If for any integer $m$, we have $G(m):=\#\{p: g(p)=2p+ m\}=O(1)$, i.e., the  function $G(m)$ is uniformly bounded, then for all sufficiently large $x$, 
          \[
    \pi_{A, \split,  g(\cdot)}(x) \ll x^{\frac{1}{2}}.
    \]
      \item\label{(2)} If $K=\Q$ and $g(p)=2p+m_0$ for a fixed integer $m_0$ (in particular, the assumption in (\ref{(1)}) is not satisfied), then for all sufficiently large $x$, 
        \[
    \pi_{A, \split,  g(\cdot)}(x) \ll \frac{x (\log\log x)^2}{(\log x)^2}.
    \]
  \end{enumerate}
  In both cases, the implicit constants in $\ll$  depend on $A$, $K$, and $g(\cdot)$. 
\end{theorem}

The proof of part (1) relies only on properties of the Frobenius polynomial  $P_{A, \fp}(X)$, which are discussed in \Cref{S2.3}.  The second part relies on the modularity of $A$, and its proof makes essential use of the result of Thorner and Zaman \cite[Theorem 1.4]{ThZa2018} on newforms of weight 2.

This result may have broad applications. For example, if $g(p)$ is a polynomial in $p$, then part (1) applies. In particular, one may take
 $g(p)$ to be constant. The power saving in $x$ allows us to derive the following  upper bound for the number of primes for which $a_{2, \fp}$ lies in a short interval.

\begin{corollary}\label{a2p-interval}
 Let $A/K$ be an absolutely simple abelian surface such that $\End_{K}(A)= \End_{\ol{K}}(A)$ and 
  $A$  has real multiplication.  Let $x>0$,  $\delta>0$, and $I \subset [-6x, 6x]$ be an interval satisfying $ |I|\leq x^{\frac{1}{2}}(\log x)^{-(1+\delta)}$.  Then  for all sufficiently large  $x$,
    \[
   \#\{\fp\in \Sigma_K:  N(\mathfrak{p}) \leq x, \fp \nmid N_A, \overline{A}_{\mathfrak{p}} \text{ splits}, a_{2, \fp}\in I\} \ll  \frac{x}{(\log x)^{1+\delta}},
    \] 
    where the implicit constant in $\ll$ depends on $A, K,$ and  $\delta$. 
\end{corollary}

\begin{remark}
We provide conjectural asymptotics of  counting function  $ \pi_{A, \split,  g(\cdot)}(x)$ in  special cases by a probabilistic argument similar to that used in the Lang–Trotter conjecture.

    Assume $\fp$ is a degree 1 prime of $K$ lying above a rational prime $p$. Then, by the Weil conjectures, we have $|a_{2, \fp}|\leq 6p$. Hence, for a fixed integer $t$, the probability that $a_{2, \fp}=t$ is  approximately $\frac{1}{p}$.  Consequently, for the constant function $g(p)=t$, we expect $ \pi_{A, \split,  g(\cdot)}(x) \ll \log\log x$. In particular, a heuristic \textit{upper bound} for \Cref{a2p-interval} is $O\left(x^{\frac{1}{2}}(\log x)^{-(\delta+1)}\log\log x\right)$.  
    
    Now suppose $g(p)=2p+t$ for a fixed integer $t$. By \Cref{main-lem-RM} (see also \cite[Section 1.2]{ShTa2020}), the probability that  $\ol{A}_\fp$ splits is roughly $\frac{1}{\sqrt{p}}$. In this setting, the condition $a_{2, \fp}=g(p)$ implies $b=t$, where $b:=b_\fp$ is the constant  in \Cref{main-lem-RM}. Since $|b|\leq 2\sqrt{p}$, the probability that $b=t$ is approximately $\frac{1}{\sqrt{p}}$. A conditional probability computation then suggests that the  probability that $\ol{A}_\fp$ splits  and $a_{2, \fp}=g(p)$ is roughly $\frac{1}{p}$. Thus, we expect $ \pi_{A, \split,  g(\cdot)}(x) \asymp  \log\log x$ holds in Case (2) of \Cref{a2p-split} as well. 
\end{remark}

 Lastly, we give an overview of the paper.  Section 2  covers essential properties of abelian surfaces over finite fields and  introduces Galois representations for abelian surfaces over number fields.  In Section 3, we introduce essential analytic ingredients for the proof, including the sieve-theoretic lemma and results on unconditional effective Chebotarev Density Theorem. In Section 4, we study the algebraic and arithmetic properties of certain low-dimensional algebraic groups, which may be of independent interest.  In Section 5, we find a  criterion (\Cref{split-lem}) for determining when the Frobenius polynomial $P_{A, \fp}(X)$ splits modulo a prime $\ell$, via an explicit condition involving the Legendre symbol. This result plays a crucial role in  the application of the sieve-theoretic arguments introduced in Section 3.   
Section 6 contains the proof of  \Cref{a2p-split} and  \Cref{a2p-interval}.

\subsection*{Notation}
Throughout, we use $p$ and $\ell$ to denote rational primes; we use $q$ to denote a prime power and denote by  $\F_q$ the finite field with $q$ elements. For a prime $\ell$ and an integer $n$, we denote by $\kronecker{n}{\ell}$ the Legendre symbol, which depends only on the residue class of $n\pmod \ell$. 

For a number field $K$, we denote by $\mathcal{O}_K$ the ring of integers of $K$ and  $\Sigma_K$ the set of all nonzero prime ideals in $\mathcal{O}_K$.  A prime ideal $\fp\in \Sigma_K$ is also referred to as a prime of $K$. We denote by $N(\fp)$ the norm of the ideal $\fp$.   We write $\F_{\fp}=\F_{N(\fp)}$ for the for the corresponding residue field.  If $N(\fp)=p$, we say $\fp$ is a prime of degree 1 of $K$. Given an integer  $N$,  we write $\fp\nmid N$ to indicate that $\fp$ does not divide the principal ideal generated by $N$. 

If $A$ is an abelian variety defined over a number field $K$, we denote by $N_A\in \Z$ the absolute norm of the conductor ideal of $A$. 

Let $X$ be a set and  $f: X\to \R$ and $g: X\to \R_{\geq 0}$ be two functions. We write $f\ll g$ if there is an absolute constant $C>0$ such that  $|f(x)|\leq C g(x)$ for all $x\in X$. We also write $f=\O(g)$ interchangeably.

For a integer $n\geq 1$ and a unitary ring $R$ with the group of units $R^{\times}$, we write $M_n(R)$ to denote the set of all $n\times n$ matrices with coefficients in $R$; we write $\GL_n(R)$ to denote the general linear group with coefficients in $R$;  we write $\GSp_{2n}(R)$ to denote the general symplectic group defined by
\[
\left\{M\in \GL_{2n}(R): J^{t}MJ=\mu M, \mu\in R^{\times}\right\},
\]
where $J :=\begin{pmatrix} 0 & I_n \\ -I_n & 0\end{pmatrix}$ and $J^{t}$ is the transpose of $J$.

\section*{Acknowledgments}
I would like to thank the discussions related to this work with Alina Carmen Cojocaru while I was at the University of Illinois Chicago. I also appreciate enlightening conversations and helpful comments from Valentin Blomer, Chantal David,  Valentijn Karemaker, Junxian Li,  Freddy Saia, Yunqing Tang, and Pengcheng Zhang.   Finally, I would like to thank the Max-Planck-Institut f\"{u}r Mathematik in Bonn for the stimulating atmosphere  to finish this research.  

\section{Preliminaries on abelian surfaces}\label{sec:abelianS}

\subsection{Supersingular abelian surfaces over finite fields}\label{S-2-1}

Let $A$ be an abelian surface  defined over $\F_q$. We begin by revisiting the properties of the  Frobenius polynomial of $A$. The Frobenius endomorphism of $A$, denoted by $\pi_q(A)$, is induced by the Frobenius automorphism $x\mapsto x^q$ on $\F_q$.   The characteristic polynomial of  $\pi_q(A)$, called the {\emph{Frobenius polynomial of $A$}}, is an integral  polynomial of the form 
\[
P_{A, q}(X)=X^4+a_{1, q}X^3+a_{2, q}X^2+qa_{1, q}X+q^2.
\]
Moreover, $P_{A, q}(X)$ is a  $q$-Weil polynomial, i.e., all of its complex roots have absolute value $\sqrt{q}$,  and  
  has the following factorization over $\C[X]$:
\[
P_{A, q}(X)= (X-\alpha)(X-\frac{q}{\alpha})(X-\beta)(X-\frac{q}{\beta}),
\]
where $\alpha, \beta\in \C$ and $|\alpha|=|\beta|=q^{\frac{1}{2}}$.
The coefficients $a_{1, q}, a_{2, q}$ of  $P_{A, q}(X)$ 
satisfy the following bounds (see \cite[Lemma 2.1, p. 323]{MaNa2002}):
\begin{equation}\label{coef-bound}
|a_{1, q}|\leq 4\sqrt{q}, \; 2|a_{1, q}|\sqrt{q}-2q\leq a_{2, q}\leq \frac{a_{1, q}^2}{4}+2q.
\end{equation}
 The constant
\[
a_q(A) :=- a_{1, q}
\]
 is referred as the \emph{Frobenius  trace} of $A$. 
We then define the {\emph{discriminant}}\footnote{This is  different from the discriminant  of a polynomial defined by the product of the squares of the differences of the  roots.} of $P_{A, q}(X)$ as the constant
\[
\Delta_{A, q} := a_{1, q}^2-4a_{2, q}+8q.
\]

The following lemma record several equivalent definitions of a supersingular abelian surface defined over the finite prime field $\F_p$. 
\begin{lemma}\label{def-ss-surface}
The following conditions are equivalent, and each may  be taken as  the definition of a supersingular abelian surface $A/\F_p$:
\begin{enumerate}
\item $A$ is isogenous over $\ol{\F}_p$ to a product of two supersingular elliptic curves. 
\item The Newton polygon of $A$ is a line segment  of slope $\frac{1}{2}$.
\item The $p$-rank 
of $A$ is 0, i.e., $A[p](\ol{\F}_p)=\{0\}$.
\end{enumerate} 
Moreover if  $p\geq 17$ then we have
 $a_p(A)=0$.
\end{lemma}
\begin{proof}
(1) can be found in \cite[p.339]{RuSi2002}. (2), (3), and their equivalence can be found in  \cite[Proposition 3.1]{Pr2019}.

To show the last part, observe that since the Newton polygon of  $A$ is $\frac{1}{2}$,  then the valuation satisfies $v_p(a_{1, p})\geq \frac{1}{2}$. So    either $|a_{1, p}|\geq p$ or $a_{1, p}=0$. From the upper bound of $a_{1, p}$ in  (\ref{coef-bound}), it follows that $|a_{1, p}|\geq p$ can only occur when  $p\leq 16$. This completes the proof. 
\end{proof}

\begin{remark}\label{lem-7-remark}
Let $A_f/\Q$ be a modular abelian variety associated to a weight 2 non-CM newform  $f$ in the context of Eichler-Shimura theory. Then for all but finitely many primes $p\nmid N_{A_f}$,  the reduction of  $A_f$ at $p$  has $p$-rank  zero if and only if the $p$-th Fourier coefficient 
  $a_p(f)=0$  \cite[Proposition 5.2, p. 62]{BaGo1997}. In particular, let $A/\Q$ be an abelian surface such that $F:=\End_{\Q}(A)\otimes_\Z \Q $ is a real quadratic field. By a result of Ribet \cite{Ri1992}  and the proof of Serre’s modularity conjecture \cite{KhWi2009-I,KhWi2009-II}, $A$ is isogenous over $\Q$ to $A_f$ for a weight 2 non-CM newform $f$ with coefficients in $f$. Hence, for all sufficiently large primes  $p$,   the reduction
of $A$ at $p$ is supersingular if and only if $a_p(A)=a_p(f)+\iota(a_p(f))=0$, where $\iota: F\to \overline{\Q}$ denotes the nontrivial Galois automorphism of $F$. 
\end{remark}

We now make the final assertion of  \Cref{def-ss-surface}  explicit in terms of the Frobenius polynomial. 

\begin{lemma}\label{ss-Weil-poly}
    Let $p\geq 7$ be a prime. Let $f(X)=X^4+a_{1, p}X^3+a_{2, p}X^2+pa_{1, p}X+p^2$ be a $p$-Weil polynomial. Then, 
    \begin{enumerate}
        \item $f(X)$ is the Frobenius polynomial of a simple supersingular abelian surface over $\F_p$ if and only if 
    \[
    f(X)\in \{X^4+pX^2+p^2, X^4-pX^2+p^2, X^4+p^2, X^4-2pX^2+p^2\}.
    \]
    \item $f(X)$ is the Frobenius  polynomial of a supersingular abelian surface that splits over $\F_p$, i.e., the abelian surface is isogenous over $\F_p$ to a product of two (not necessarily distinct) elliptic curves over $\F_p$,  if and only if 
    \[
    f(X)= X^4+2pX^2+p^2.
    \]
    \end{enumerate}
   
\end{lemma}
\begin{proof}
Part (1) follows from \cite[Corollary 2.11, p. 324]{MaNa2002}. 

For part (2),\footnote{Part (2) could be derived by considering  possible  slopes of the Newton polygon for the Frobenius polynomial of  supersingular abelian surfaces that splits. Here, we will give a more direct proof.} the ``only if'' direction follows from the Honda–Tate theory and the classification of Frobenius polynomials for supersingular elliptic curves over $\F_p$ \cite[Theorem 4.1 (5), p.536]{Wa1969}. 
For the converse, observe from \cite[Theorem 2.9 and Corollary 2.10]{MaNa2002}  that $f(X)$ corresponds to a non-ordinary (i.e., $p\mid a_{2, p}$)  abelian surface $A/\F_p$. Since $A$    splits over $\F_p$ by assumption, it  has to be  supersingular based on the $p$-rank considerations. 
\end{proof}


\subsection{Galois  representations of abelian surfaces over number fields}\label{S2-Galois-rep}
Let $K$ be a number field and $A$ be an absolutely simple abelian surface defined over  $K$. Denote by $D := \End_{\ol{K}}(A)\otimes_{\Z}\Q$ the endomorphism algebra of $A$ over $\ol{K}$.
For a rational prime  $\ell$,   we denote by $A[\ell]$ and $T_{\ell}(A)$ the $\ell$-torsion group of $A$ and the $\ell$-adic Tate module of $A$, respectively. The actions of the absolute Galois group $\Gal(\ol{K}/K)$  on  these modules  give rise to the residual modulo $\ell$ and the $\ell$-adic Galois representations of $A$, denoted respectively by
 \[
 \ol{\rho}_{A, \ell}: \Gal(\ol{K}/K) \to \Aut_{\Z/\ell\Z}(A[\ell]) \quad \text{ and } \quad
\rho_{A, \ell}: \Gal(\ol{K}/K) \to \Aut_{\Z_\ell}(T_\ell(A)).
 \]
 Since $A$ is isogenous over $\ol{K}$ to a principally polarized abelian surface, we obtain a perfect pairing on $A[\ell]$  coming from   the   Weil paring. Hence, the image of these representations lies in the symplectic similitude groups $\GSp_4(\F_\ell)$ and $\GSp_4(\Z_\ell)$, 
 respectively, for all sufficiently large primes $\ell$. 

We now summarize known results on the images of these Galois representations, which depend on the structure of the endomorphism algebra   $D$.  We have that for all sufficiently large prime $\ell$, 
\begin{enumerate} 
    \item \label{Z-Galois-im} if $D=\Q$, then    (see, e.g, \cite{Se2013})  
    \[
    \ol{\rho}_{A, \ell}(\Gal(\ol{K}/K)) = \GSp_4(\F_\ell) \ \text{ and } \ \rho_{A, \ell}(\Gal(\ol{K}/K)) = \GSp_4(\Z_\ell);
    \]
    \item \label{RM-Galois-im} if $\End_{\ol{K}}(A)=\End_{K}(A)$ and $D=F$ is a real quadratic field such that $\ell$ splits completely in $F$ and is unramified in $K$, then    (see \cite[Theorem (5.3.5), p. 800]{Ri1976} or \cite[Remark 1.6, p. 29]{Lom16})
    \[
    \ol{\rho}_{A, \ell}(\Gal(\ol{K}/K)) = \GL_2(\F_\ell)\times_{\det}\GL_2(\F_\ell)  \ \text{ and } \
    \rho_{A, \ell}(\Gal(\ol{K}/K)) = \GL_2(\Z_\ell)\times_{\det}\GL_2(\Z_\ell),
    \]
 where 
   \[
\GL_2(R)\times_{\det}\GL_2(R)\coloneqq  \left\{(M_1, M_2) \in \GL_2(R)\times\GL_2(R): \det(M_1) =  \det(M_2)\right\}
 \]
  for $R\in \{\F_\ell, \Z_\ell\}$;
    \item \label{QM-Galois-im} if  $\End_{\ol{K}}(A)=\End_{K}(A)$ and $D$ is an indefinite quaternion algebra and  $\ell>7$ is a prime satisfying   $\ell$ splits $D$ (i.e., $D\otimes_\Q \Q_\ell \simeq M_2(\Q_\ell)$) and  is unramified in $K$, then  (see \cite[Theorem 3.7]{Oh1974} and \cite{DiRo2004})
      \[
    \ol{\rho}_{A, \ell}(\Gal(\ol{K}/K)) \simeq  \GL_2(\F_\ell)  \ \text{ and } \
    \rho_{A, \ell}(\Gal(\ol{K}/K)) \simeq \GL_2(\Z_\ell).
    \]
    In fact,  the images of $ \ol{\rho}_{A, \ell}$ and $\rho_{A, \ell}$ are diagonal embeddings of $\GL_2(\F_\ell)$ and $\GL_2(\Z_\ell)$ into the fiber products $\GL_2(\F_\ell)\times_{\det}\GL_2(\F_\ell)$ and $\GL_2(\Z_\ell)\times_{\det}\GL_2(\Z_\ell)$, respectively.
    
\end{enumerate}

Finally, we  introduce the field $K^{\conn}$ and the algebraic group $G_{A, \ell}^{\zar}$ which will play a role in the proof of \Cref{RM-poly} in the next section. 

Let  $G_{A, \ell}^{\zar}$ denote the Zariski closure of the image of $\rho_{A,\ell}$ in  $\GL_{4/\Q_\ell}$  and  $(G_{A, \ell}^{\zar})^0$ be  the connected component of  $G_{A, \ell}^{\zar}$. By a result of Serre \cite[Proposition (6.14), p. 623]{LaPi1992}, there exists a number field $K^{\conn}$, independent of the choice of $\ell$, such that the following map is an isomorphism:
\[
\Gal(\overline{K}/K^{\conn})\xrightarrow{\rho_{A, \ell}} G_{A, \ell}^{\zar} \twoheadrightarrow G_{A, \ell}^{\zar}/(G_{A, \ell}^{\zar})^0.
\]
 In other words, $K^{\conn}$ is the minimal subfield of $\ol{K}$ such that the Zariski closure of $\rho_{A, \ell}(\Gal(\ol{K}/K^{\conn}))$ is connected for all primes $\ell$. Moreover, 
if $A$ is an abelian surface, then the field $K^{\conn}$ is also the minimal field of definition for the endomorphisms of $A$. 
Hence, if $K$ is minimal so that $\End_K(A)=\End_{\ol{K}}(A)$, then $K=K^{\conn}$ (see \cite{LaPi1997} or \cite[p. 892, Lemma 2.8 and the paragraph above]{Lo2019}).

\subsection{Split reductions of abelian surfaces with real multiplication}\label{S2.3}

Let $K$ be a number field, and let $A$ be an abelian surface defined over $K$. In this section,   we  study various properties of $A$, focusing on the case where $F := \End_{\ol{K}}(A)\otimes_\Z \Q$ is a real quadratic field. We also use notation introduced in previous sections. We adopt the notation established in previous sections. The results presented here will be used to prove the first case of \Cref{a2p-split}.

  For a prime $\fp$ of $K$ such that $\fp\nmid N_A$, we denote  by $\ol{A}_\fp$ the reduction of $A$ at $\fp$. Let $q=N(\fp)$ and $p$ be the rational prime below $\fp$. The characteristic polynomial of the Frobenius endomorphism  $\ol{A}_{\fp}$ is denoted by
\begin{equation}\label{poly-standard}
P_{A, \fp}(X) := P_{\ol{A}_{\fp}, q}(X)=X^4+a_{1, \fp}X^3+a_{2, \fp}X^2+qa_{1, \fp} X +q^2 \in \Z[X].
\end{equation}
 
 We recall that for a degree 1 prime $\fp$ of $K$, the reduction  $\overline{A}_{\fp}/\F_p$  splits if it is isogenous over $\F_p$ to a product of elliptic curves over $\F_p$.

\begin{lemma}\label{main-lem-RM}
Let $A/K$ be an abelian surface with  $F= \End_{\ol{K}}(A)\otimes_\Z \Q$ being  a real quadratic field. For any degree 1 prime ideal $\fp$ of $K$ such that $\fp\nmid N_A$, if $\overline{A}_{\fp}$ splits, then 
\[
P_{A, \fp}(X) \in \left\{(X^2+bX+p)^2, (X^2+bX+p)(X^2-bX+p) \right\}
\]
for some integer $b$ such that $|b|\leq 2\sqrt{p}$.
Moreover, if $\End_{\ol{K}}(A) = \End_{K}(A)$ and $\overline{A}_{\fp}$ splits, then $\overline{A}_{\fp}$  is isogenous over $\F_p$ to the square of an elliptic curve.  In particular,  we have 
 $P_{A, \fp}(X) =(X^2+bX+p)^2$ for some integer $b$ such that $|b|\leq 2\sqrt{p}$.
\end{lemma}
\begin{proof}
See \cite[Lemma 17]{Wa2023}.
\end{proof}

In contrast, when $\overline{A}_\fp$ does not split, the Frobenius polynomial admits a factorization over $F$ as described in the following lemma. 

\begin{lemma}\label{RM-poly}
Let $A/K$ be an abelian surface with $F= \End_{\ol{K}}(A)\otimes_\Z \Q$ being   a real quadratic field.  
Moreover, we assume $\End_{\overline{K}}(A)=\End_{K}(A)$.  Then for each prime ideal $\fp$ with  $\mathfrak{p}\nmid N_A$ and denoting  $q:=N(\fp)$, we have the following factorization of $P_{A, \mathfrak{p}}(X)$ in $F[X]$:
\[
P_{A, \mathfrak{p}}(X)=(X^2+bX+q)(X^2+ \iota(b)X+q),
\]
where $b\in \mathcal{O}_F$ and $\iota: F\to \overline{\Q}$ is the nontrivial Galois automorphism of $F$.
\end{lemma}
\begin{proof}
The proof is similar to the argument in   \cite[Proposition 3.5, p. 902]{Lo2019}.  
If   $\End_{\overline{K}}(A)=\End_{K}(A)$, then  we have  $\rho_{A, \ell}(\Frob_{\fp})\in G_{A, \ell}^{\zar}=(G_{A, \ell}^{\zar})^0\subseteq \GL_2(F\otimes_\Q \Q_\ell)$ for each prime $\ell$  (see  \Cref{S2-Galois-rep}). Therefore,  for $\fp\nmid \ell N_A$, the characteristic polynomial of $\rho_{A, \ell}(\Frob_{\fp})  \subseteq \GL_4(\Z_\ell)$ is of the form $P_{A, \mathfrak{p}}(X)=f(X)\iota(f(X))$, where $f(X)$ is the characteristic polynomial of $\rho_{A, \ell}(\Frob_{\fp})$, viewed as a matrix in  $\GL_2(F\otimes_{\Q} \Q_\ell)$.
\end{proof}

\section{Preparations for analytic theory }\label{sec:analitic}
\subsection{An inclusion-exclusion lemma}\label{S3.1}

Let  $x\in \R_{>0}$  and $t=t(x)\in \Z_{>0}$  be a function tending to infinity as $x\to \infty$.  We also assume
\begin{equation}\label{t-bound}
t(x)\ll (\log x)^{\frac{1}{2}}.
\end{equation}
Let 
$\cal{M}$ be a subset of rational primes up to $x$. The goal of this section is to give an upper bound of $\cal{M}$ using an inclusion-exclusion principle.

We consider a set $\cal{P}$ consisting of $t$ primes that depends on $x$:
  \begin{equation}\label{P-size}
    \ell_1< \ell_2 < \ldots <\ell_t \ll \frac{\log x}{\log\log x}.   
  \end{equation}
  For  each $\ell\in \cal{P}$, we denote by 
  \[
  \cal{M}_\ell =\{p\in \cal{M}: p \pmod \ell \not\in \Omega_\ell\}, 
  \]
  where 
  \[
\Omega_\ell: =\{n\pmod \ell: \kronecker{n}{\ell}=-1\}.
\]
It is an easy observation that $|\Omega_\ell|=\frac{\ell-1}{2}$.
Moreover, we set $\displaystyle P_t:=\prod_{\ell\in \cal{P}}\ell$ and 
\[
\Omega_{P_t} := \{n\pmod{P_t}: \kronecker{n}{\ell}=-1 \ \text{ for all } \ell\in \cal{P}\}.
\]
Then by the Chinese Remainder Theorem, we have $\displaystyle |\Omega_{P_t}|=\prod_{\ell\in \cal{P}}\frac{\ell-1}{2}$.
Therefore, for the set 
\[
\cal{S} :=\cal{M}\backslash \bigcup_{\ell\in \cal{P}} \cal{M}_\ell = \{p\in \cal{M}: \kronecker{p}{\ell}=-1 \  \text{ for all } \ell\in \cal{P}\}
\]
we have that for all sufficiently $x$,
\begin{align*}
 |\cal{S}| & \leq  \#\{p\leq x: p\pmod \ell \in \Omega_\ell \ \text{ for all } \ell\in \cal{P} \}\\
  & =  \#\{p\leq x: p\pmod{P_t} \in \Omega_{P_t}\}\\
  & = \sum_{a\pmod {P_t} \in \Omega_{P_t}}  \#\{p\leq x: p\equiv a \pmod{P_t}\}\\
  & \leq \displaystyle \prod_{\ell\in \cal{P}}\left( \frac{\ell-1}{2}\right) \cdot  \frac{2x}{\displaystyle  \log (x/P_t)\prod_{\ell\in \cal{P}}(\ell-1)}\\
 & \ll \frac{x}{2^t \log(x/P_t)},
  \end{align*}
where in the second last step we use the Brun–Titchmarsh Theorem, which is applicable because $P_t<x$  for all sufficiently large $x$ by (\ref{t-bound}).

We summarize the result in the following lemma. 
\begin{lemma}\label{uncond-lem}
    With the notation above, we have  for all sufficiently large  $x$,
    \[
    |\cal{M}|\ll \sum_{1\leq j\leq t}|\cal{M}_{\ell_j}|+ \frac{x}{2^t \log(x/P_t)}.
    \]
\end{lemma}

  \begin{proof}
By using an inclusion-exclusion principle and  the bound of $|\cal{S}|$, we immediately obtain  
\[
 |\cal{M}|\leq \sum_{1\leq j\leq t}|\cal{M}_{\ell_j}|+ |\cal{S}|\ll \sum_{1\leq j\leq t}|\cal{M}_{\ell_j}|+ \frac{x}{2^t \log(x/P_t)}.
\]
\end{proof}

 \begin{remark}
From the proof of  \Cref{uncond-lem}, it is clear that bounding  $|\cal{M}|$ is sufficient to control $|\cal{S}|$. In the setting  of sieve theory, the set $\cal{S}$ can be regarded as a sifted set (usually denoted by $\cal{S}(\cal{M}, \cal{P}, (\Omega_\ell)_{\ell\in \cal{P}})$).
 By applying the P\'olya-Vinogradov inequality  and following the arguments in   \cite[Proof of Lemma 4.1, pp. 264-265]{Wa1990}, we get 
\[
|\cal{S}|\leq    \frac{x}{2^t} +\sqrt{P_t} \log P_t.
\]
Alternatively, we can also apply the  large sieve in (\cite[Theorem 7.10, p. 180]{IwKo2004}), we get 
    \[
    |\cal{S}|\leq \frac{x}{\displaystyle 2^t\left(1+\sum_{1\leq i\leq t}\frac{1}{\ell_i}\right)}+ \frac{P_t^2}{\left(\displaystyle 1+\sum_{1\leq i\leq t}\frac{1}{\ell_i}\right)}.
    \]
    However, under the present assumptions on $\cal{P}$ and $t$, neither of these alternative bounds improves the bound in   \Cref{uncond-lem}.
 \end{remark}

\subsection{Effective Chebotarev Density Theorem}\label{ECDT}
 Let $K$ be a number field.  We denote by $d_K$   the absolute discriminant of $K/\Q$ and $n_K:= [K:\Q]$. Let $L/K$ be a finite  extension of number fields with Galois group $G :=\Gal(L/K)$.  We set
\begin{align*}
\cal{P}(L/K) & :=\{p: \text{ there exists }  \mathfrak{p}\in \Sigma_K, \text{ such that } \mathfrak{p}\mid p \text{ and $\mathfrak{p}$ is ramified in } L\},
\\
M(L/K) & := [L:K] d_K^{\frac{1}{n_K}}\prod_{p\in \cal{P}(L/K)} p
\end{align*}
Let $\cal{C}\subseteq G$ be a conjugation invariant  set  of $G$. We denote by
\[
\pi_{\cal{C}}(x, L/K) := \#\{\fp\in \Sigma_K: N(\fp)\leq x, \fp \text{ unramified in } L/K, \kronecker{L/K}{\fp}\subseteq \cal{C}\},
\]
where $\kronecker{L/K}{\fp}$ is the Artin symbol of $L/K$ at $\mathfrak{p}$.

First, we state the upper bound of $\log d_L$ due to Hensel, where the proof is given in  \cite[Proposition 5, p. 129]{Se1981}. 
\begin{lemma} \label{lem-logdK} If $L/K$ is a finite Galois extension of number fields, then
\begin{align*}
\log d_L  \leq & [L : K]\log d_K +  
([L : \Q] - [K : \Q])
\log \rad(d_{L/K})+ [L : \Q]  \log [L : K],
\end{align*} 
 where $\rad n \coloneqq \prod_{p \mid n} p$ is the radical of the integer $n$ and $d_{L/K}$ is the relative discriminant for $L/K$. 
\end{lemma}

Next, we recall an unconditional version of the  effective Chebotarev Density Theorem due to Thorner and Zaman \cite[Theorem 9.1, p. 5022]{ThZa2018}.
\begin{theorem}\label{uncond-ECDT}
Let $L/K$ be an abelian extension of number fields. Let $G=\Gal(L/K)$ and $\cal{C}$ be a conjugation invariant set of $G$. Then,  for $x$ with
$\log x\gg n_K \log (M(L/K)x)$,  we have 
\[
\pi_{\cal{C}}(x, L/K)\ll \frac{|\cal{C}|}{|G|}\Li(x).
\]
\end{theorem}

Finally, we present the following induction and restriction properties of $\pi_{\cal{C}}(x, L/K)$.

\begin{lemma}\label{func-ECDT}
Let  $L/K$ be a finite extension of  number fields with Galois  group $G$. Let $\cal{C}$ be a conjugation invariant set of $G$. Let $H$ be a subgroup of $G$ and $N$ be a normal subgroup of $H$. Assume that 
\begin{enumerate}
    \item every element of $\cal{C}$  is conjugate over $G$ to an element in $H$.\label{H-C}
    \item $N(\cal{C} \cap H)\subseteq \cal{C}\cap H$.\label{N-C}
\end{enumerate}
   Then, we have that for all sufficiently large $x$, 
  \begin{align*}
          \pi_{\cal{C}}(x,L/K)\ll \pi_{\ol{\cal{C}\cap H}}(x,L^N/L^H)+\O\left(n_{L^H} \left(\frac{x^{\frac{1}{2}}}{\log x}+\log M(L/K)\right)\right),
    \end{align*}
    where $L^H$  and $L^N$ represent the fixed field of $L$ by $H$ and $N$, respectively, and  $\ol{\cal{C}\cap  H}$ represents the image of $\cal{C}\cap  H$ in $H/N$.
\end{lemma}
\begin{proof}

By \cite[Lemma 2.6 (i), p. 241  and Lemma 2.7, p. 242]{Zy2015},  we have 
\begin{align*}
&  \pi_{\cal{C}}(x,L/K)+\O\left(n_K \left(\frac{x^{\frac{1}{2}}}{\log x}+\log M(L/K) \right)\right) \\
 & \leq  \pi_{\cal{C}\cap H}(x,L/L^H)+\O\left(n_{L^H} \left(\frac{x^{\frac{1}{2}}}{\log x}+\log M(L/L^H) \right)\right).
\end{align*}
Therefore, 
 \begin{align*}
         \pi_{\cal{C}}(x,L/K)& \ll \pi_{\cal{C}\cap H}(x,L/L^H)+ n_K \left(\frac{x^{\frac{1}{2}}}{\log x}+\log M(L/K) \right)\\
         & \hspace{3.2cm} + 
         n_{L^H} \left(\frac{x^{\frac{1}{2}}}{\log x}+\log M(L/L^H) \right).
    \end{align*}
Similarly, by
\cite[Lemma 2.6 (ii), p. 241  and Lemma 2.7, p. 242]{Zy2015}, we have 
 \begin{align*}
        \pi_{\cal{C}\cap H}(x,L/L^H)&= \pi_{\ol{\cal{C}\cap H}}(x,L^N/L^H) +  n_{L^H} \left(\frac{x^{\frac{1}{2}}}{\log x}+\log M(L/L^H) \right)\\
         & \hspace{3.3cm} + 
         n_{L^H} \left(\frac{x^{\frac{1}{2}}}{\log x}+\log M(L^N/L^H) \right).
    \end{align*}
The conclusion  follows by combining the two results together. 
\end{proof}

\section{Results on matrix groups}\label{sec: matrix}
\subsection{Subgroups and conjugacy classes of \texorpdfstring{$\GL_2(\F_\ell)$}{Lg}}\label{QM-section}
  Let $\ell$ be a rational prime. In this section, our attention is directed towards subsets of the general linear group $\GL_2(\F_{\ell})$. We introduce the following subgroups of $\GL_2(\F_{\ell})$.
  
\begin{align*}
  \cal{B}(\ell) & :=
    \left\{M\in  \GL_2(\F_\ell): M=\left(\begin{matrix}
    \lambda_1 & a\\ 
    0 & \lambda_2
    \end{matrix}\right), \lambda_1, \lambda_2 \in \F_\ell^{\times}, a \in \F_\ell \right\},\\
    \cal{U}(\ell) & := 
    \left\{M\in  \GL_2(\F_\ell): M=\left(\begin{matrix}
    1 & a\\ 
    0 & 1
    \end{matrix}\right), a\in \F_\ell \right\},\\
    \cal{U}'(\ell) & := 
    \left\{M\in  \GL_2(\F_\ell): M=\left(\begin{matrix}
    \lambda & a\\ 
    0 & \lambda
    \end{matrix}\right), \lambda \in \F_\ell^{\times}, a\in \F_\ell \right\}, \\
    \cal{T}(\ell) & :=\left\{M\in  \GL_2(\F_\ell): M=\left(\begin{matrix}
    \lambda_1 & 0\\ 
    0 & \lambda_2
    \end{matrix}\right), \lambda_1, \lambda_2 \in \F_\ell^{\times} \right\}.
\end{align*}

We recall several properties of $\GL_2(\F_\ell)$.

\begin{proposition}\label{group-prop-QM}  
Let $\ell\geq 5$. The following statements hold. 
    \begin{enumerate}
     \item 
        \begin{align*}
         |\GL_2(\F_\ell)| &= (\ell-1)^2\ell(\ell+1),\\
            |\cal{B}(\ell)| & = (\ell-1)^2\ell,\\
            |\cal{U}(\ell)| &= \ell,\\
            |\cal{U}'(\ell)| & = \ell(\ell-1),\\
            |\cal{T}(\ell)| & = (\ell-1)^2.
        \end{align*}
        \label{group-prop-3-QM}
        \item $\cal{U}(\ell)$ and $\cal{U}'(\ell)$ are normal subgroups of $\cal{B}(\ell)$. \label{group-prop-1}
        \item The quotients $\cal{B}(\ell)/\cal{U}(\ell)$ and $\cal{B}(\ell)/\cal{U}'(\ell)$ are abelian. Moreover, we have the isomorphism   $\cal{B}(\ell)/\cal{U}(\ell)\simeq \cal{T}(\ell)$.\label{group-prop-2-QM}
        \end{enumerate} 
\end{proposition}
\begin{proof}
The proofs are straightforward and can be found in \cite[Section 4]{CoWa2023}.
\end{proof}

Next, we introduce 
some conjugation invariant sets of $\GL_2(\F_\ell)$.
\begin{align}
     \cal{C'}^{4}(\ell) & := 
    \{M\in \GL_2(\F_\ell): \car_M(X) = X^2-\mu =\prod_{1\leq j\leq 2}(X-\lambda_j), \lambda_j\in \F_\ell^{\times}, \mu \in \F_\ell^{\times}\},\label{conj-6}\\
         \cal{C'}^{5}(\ell) & := 
    \{M\in \GL_2(\F_\ell): \car_M(X) = X^2+\mu=\prod_{1\leq j\leq 2}(X-\lambda_j),\lambda_j\in \F_\ell^{\times}, \mu \in \F_\ell^{\times}\}.\label{conj-7}
\end{align}
\begin{remark}\label{conj-class-GSp}
Although the sets defined in  (\ref{conj-6}) and (\ref{conj-7})  are identical as subsets of $\GL_2(\F_\ell)$, we distinguish them notationally to streamline the argument presented in \Cref{proof-QM}.
\end{remark}

\begin{proposition}\label{conj-prop-QM} 
Let $\ell\geq 5$. For each $i\in\{4, 5\}$, the following statements hold.
  \begin{enumerate}
    \item $\cal{C'}^{i}(\ell)$ is nonempty and is invariant under conjugation by $\GL_2(\F_\ell)$. \label{conj-prop-1-QM}
    \item Every element of $\cal{C'}^{i}(\ell)$ is conjugate over $\GL_2(\F_\ell)$ to an element in $\cal{B}(\ell)$. \label{conj-prop-2-QM}
    \item We have $\cal{U}'(\ell) \cal{C'}^{i}(\ell)\subseteq \cal{C'}^{i}(\ell)$. \label{conj-prop-3-QM}
    \end{enumerate}
\end{proposition}
\begin{proof}
Since $\cal{C'}^4(\ell)=\cal{C'}^5(\ell)$, it suffices to give a proof for $i=4$. 

For (1), we take $\mu\in \F_\ell^{\times}$ such that $\mu=\lambda^2$  for some $\lambda\in \F_\ell^{\times}$. Then,  the matrix $\left(\begin{matrix}
    -\lambda & 0 \\
    0& \lambda
\end{matrix}\right)$  is an element in $\mathcal{C'}^4(\ell)$. So $\mathcal{C'}^4(\ell)$ is nonempty. The set $\cal{C'}^4(\ell)$ is  invariant under conjugation over $\GL_2(\F_\ell)$  since eigenvalues are invariant under conjugation.

(2) is a basic fact of the Jordan normal form of matrices in $\GL_2(\F_\ell)$.

 To prove part (3), we take an  element $M\in \cal{C'}^4(\ell)$. Then for any $N\in \cal{U}'(\ell)$ with diagonal entries equal to the same value $\lambda\in \F_\ell^{\times}$, we have
    \begin{align*}
      \car_{NM}(X)& = X^2 +\mu \lambda^2 = \prod_{1\leq j\leq 2}(X-\lambda\lambda_j).
    \end{align*}
   Therefore,  we have  $\cal{U}'(\ell) \cal{C'}^4(\ell)\subseteq \cal{C'}^4(\ell)$. 
\end{proof}

Finally, we set
\begin{align*}
\cal{C'}^{i}_{\mathcal{B}}(\ell) & :=\cal{C'}^i(\ell)\cap \mathcal{B}(\ell), \  i\in \{4, 5\},\\
\ol{\cal{C'}^{i}_{\mathcal{B}}(\ell)} & :=\text{image of } \cal{C'}^i_{\mathcal{B}}(\ell) \text{ in } \mathcal{B}(\ell)/\cal{U}'(\ell), \  i\in \{4, 5\}.
\end{align*}

\begin{proposition}\label{conj-size-QM}
        We have $\left|\ol{\cal{C'}^{i}_\cal{B}}(\ell)\right| \ll 1$ for $i\in \{4, 5\}$.
 \end{proposition}
\begin{proof}
This follows from \cite[Lemma 17 (iv), p. 701]{CoWa2023}.
\end{proof}

\subsection{Subgroups and conjugacy classes of \texorpdfstring{$\GSp_4(\F_\ell)$}{Lg}}\label{triv-end-group}
Fix a prime $\ell>5$. We consider specific subgroups and conjugation invariant subsets of the general symplectic group $\GSp_4(\F_\ell)$ that  naturally arise as  images of the residual modulo $\ell$ Galois representations associated to an absolutely simple abelian surface $A/K$, satisfying $\End_{\ol{K}}(A)\otimes_{\Z} \Q=\Q$. In particular,  these objects  correspond to the Galois subgroups of $K(A[\ell])/K$  and the associated conjugation invariant sets, where $K(A[\ell])$ is the $\ell$-division field of $A$.  

 We recall the block matrix definition:
\[
\GSp_{4}(\F_\ell)
=
\left\{
\begin{array}{lll}
                                  &  &-C^{t} A + A^{t} C = 0
\\
\begin{pmatrix} A & B \\  C & D \end{pmatrix} \in \GL_{4}(\F_\ell): & 
A, B, C, D \in M_2(\F_\ell), \mu \in \F_{\ell}^{\times}, & 
- C^{t}B+A^tD=\mu I 
\\
                                    & & -D^tB+B^tD=0
\end{array}\right\}
\]
 and introduce the following subsets of $\GSp_4(\F_\ell)$: 
\begin{align*}
 GB(\ell) &:=
   \left\{
  \begin{pmatrix} 
 A & \mu^{-1}AS\\
 0 & \mu (A^t)^{-1}
  \end{pmatrix}
  \in \GL_{4}(\F_\ell): 
  A \in \GL_2(\F_\ell) \text{ is an upper triangular matrix},  \right.\\
  & \hspace{5.5cm} \left.
  S \in M_2(\F_\ell) \text{ is a symmetric matrix}, \mu \in \F_{\ell}^{\times} \right\}, 
  \\
   GU(\ell) &:=
   \left\{
  \begin{pmatrix} 
 A  & AS\\
 0 & (A^t)^{-1}\\
  \end{pmatrix}
  \in \GL_{4}(\F_\ell): 
  A = \left(\begin{matrix}
    1 & a\\ 0 & 1
    \end{matrix}\right) \ a\in \F_\ell, \right.\\
    & \hspace{5.4cm} \left.
  S \in M_2(\F_\ell) \text{ is a symmetric matrix}\right\},
 \\
GU'(\ell)
 &:=
  \left\{
  \begin{pmatrix} 
 A & \mu^{-1}AS\\
 0 & \mu (A^t)^{-1}
 \end{pmatrix}
 \in \GL_{4}(\F_\ell): 
 A =\left(\begin{matrix}
    \lambda & a\\ 0 & \lambda
    \end{matrix}\right),
      a\in \F_\ell, \lambda\in \F_\ell^{\times}, \mu= \lambda^2, \right.\\
      & \hspace{5.5cm} \left.
  S \in M_2(\F_\ell) \text{ is a symmetric matrix} \right\},
  \\
   GT(\ell)
  &:=
   \left\{
  \begin{pmatrix} 
 A  & 0 \\
 0 &  \mu A^{-1} 
  \end{pmatrix}
  \in \GL_{4}(\F_\ell): 
  A \in \GL_2(\F_\ell) \text{ is diagonal},
  \mu\in \F_{\ell}^{\times} \right\}.
  \end{align*}
It is easy to check that they are all subgroups of $\GSp_4(\F_\ell)$ (see \cite[Proposition 8, p. 14]{CoWa2022} for a proof).  The following properties of these groups will be used. 
\begin{proposition}\label{gsp-prop}
Let $\ell\geq 5$. The following statements hold.
\begin{enumerate}
    \item     
    \begin{align*}
    |\GSp_4(\F_\ell)|& = (\ell-1)^3 \ell^{4}(\ell+1)^2(\ell^{2}+1),\\
            |GB(\ell)| &=  \ell^{4}(\ell-1)^{3},\\
            |GU(\ell)| &= \ell^4,\\
            |GU'(\ell)| & = \ell^4(\ell-1), \\
            | GT(\ell)| & = (\ell-1)^3.
        \end{align*}
 \label{group-prop-6}
 \item $GU(\ell)$ and $GU'(\ell)$ are normal subgroups of $GB(\ell)$. \label{group-prop-4}
 \item The quotients $GB(\ell)/GU(\ell)$ and $GB(\ell)/GU'(\ell)$ are abelian. Moreover, we have the isomorphism $GB(\ell)/GU(\ell)\simeq GT(\ell)$.\label{group-prop-5}
 \item $GB(\ell)$ is a Borel subgroup of $\GSp_4(\F_\ell)$,  $GU(\ell)$ is a unipotent subgroup of $\GSp_4(\F_\ell)$, and
 $GT(\ell)$ is a maximal torus of $\GSp_4(\F_\ell)$. \label{group-prop-7}
\end{enumerate}
\end{proposition}
\begin{proof}
See \cite[Proposition 8, p. 14, Proposition 9, p. 16, and Proposition 11, p. 17]{CoWa2022} for part (\ref{group-prop-6})--(\ref{group-prop-5}). Part (4) follows from the fact that  Borel subgroups (resp. unipotent subgroups and maximal torus) of $\GSp_4(\F_\ell)$ are conjugate with each other. Then, we can compare the  cardinality of $|GB(\ell)|$, $|GU(\ell)|$, and $|GT(\ell)|$ with the ``standard" ones such as  those in \cite[p. 308]{Br2015}.
\end{proof}


Next, we define several conjugation invariant sets of $\GSp_4(\F_\ell)$.
\begin{align}
    G\cal{C}^1(\ell) & := 
    \left\{M\in \GSp_4(\F_\ell): \car_M(X) = X^4+\mu X^2+\mu^2\right. \nonumber \\
    & \hspace{5.45cm} =\prod_{1\leq j\leq 2}(X-\lambda_j)(X-\mu \lambda_j^{-1}),  \mu, \lambda_j \in \F_\ell^{\times}\}, \label{G-conj-1}\\
    G\cal{C}^2(\ell) & := 
    \left\{M\in \GSp_4(\F_\ell): \car_M(X) = X^4-\mu X^2+\mu^2 \right. \nonumber \\
    & \hspace{5.45cm} =\prod_{1\leq j\leq 2}(X-\lambda_j)(X-\mu \lambda_j^{-1}), \mu, \lambda_j \in \F_\ell^{\times}\}, \label{G-conj-2}\\
    G\cal{C}^3(\ell) & := 
    \left\{M\in \GSp_4(\F_\ell): \car_M(X) = X^4  +\mu^2 \right. \nonumber \\
    & \hspace{5.45cm} =\prod_{1\leq j\leq 2}(X-\lambda_j)(X-\mu \lambda_j^{-1}), \mu, \lambda_j \in \F_\ell^{\times}\}, \label{G-conj-3}\\
     G\cal{C}^4(\ell) & := 
    \left\{M\in \GSp_4(\F_\ell): \car_M(X) = (X^2-\mu)^2 \right. \nonumber \\
    & \hspace{5.45cm} =\prod_{1\leq j\leq 2}(X-\lambda_j)(X-\mu \lambda_j^{-1}), \mu, \lambda_j \in \F_\ell^{\times}\}, \label{G-conj-4}\\
    G\cal{C}^5(\ell) & := 
    \left\{M\in \GSp_4(\F_\ell): \car_M(X) = (X^2+\mu)^2 \right. \nonumber \\
    & \hspace{5.45cm} =\prod_{1\leq j\leq 2}(X-\lambda_j)(X-\mu \lambda_j^{-1}), \mu, \lambda_j \in \F_\ell^{\times}\}.\label{G-conj-5} 
\end{align}
 Similar to  \Cref{conj-class-GSp}, the sets (\ref{G-conj-1}) and (\ref{G-conj-2}) (resp. (\ref{G-conj-4}) and (\ref{G-conj-5})) are  the same.  Distinguishing their name will make the argument   in  \Cref{thm-proof-triv}  easier.

\begin{proposition}\label{gsp-conj-prop}
  Let $\ell$ be an odd prime such that 
$\kronecker{-1}{\ell}=\kronecker{2}{\ell}=\kronecker{3}{\ell}=1
$. For each $1\leq i\leq 5$, the following statements hold.  
  \begin{enumerate}
    \item $G\cal{C}^i(\ell)$ is nonempty and  is  invariant under conjugation in  $\GSp_4(\F_\ell)$. \label{conj-prop-4}
    \item Each element of $G\cal{C}^i(\ell)$ is conjugate over $\GSp_4(\F_\ell)$ to some element of $GB(\ell)$.  \label{conj-prop-5}
    \item We have $GU'(\ell) G\cal{C}^i(\ell)\subseteq G\cal{C}^i(\ell)$. \label{conj-prop-6}
\end{enumerate}
\end{proposition}
\begin{proof}
   For (1), the conjugation invariant properties are obvious since eigenvalues of any element in $\GSp_4(\F_\ell)$ is invariant under conjugation. 
   
   For each  $1\leq i\leq 5$ and $\ell$ satisfying the assumption of this proposition, we will prove the nonemptiness of $G\cal{C}^i(\ell)$ by constructing an element explicitly. To show the nonemptiness of $G\cal{C}^1(\ell)$, we choose $\mu\in \F_\ell^{\times}$ such that there exists an element $a\in \F_\ell^{\times}$ satisfying $a^2=\mu$.   Then, we  find  matrices  $M_1, M_2\in \GL_2(\F_\ell)$ whose characteristic polynomials  are $X^2+aX+\mu$ and $X^2-aX+\mu$, respectively. These polynomials split over $\F_\ell$ since $\kronecker{-3}{\ell}=1$ and
$
a^2-4\mu=-3\mu 
$
is a square. 
Therefore,  the matrix
\[
\begin{pmatrix} 
M_1 & 0 \\ 0 & M_2\end{pmatrix}
\]
is in $G\mathcal{C}^1(\ell)$.

For the nonemptiness of $G\cal{C}^2(\ell)$, we choose $\mu\in \F_\ell^{\times}$ such that there are $a, b\in \F_\ell^{\times}$ satisfying $a^2=3\mu$ and $b^2=-\mu$.  Then, we can find matrices $M_1, M_2\in \GL_2(\F_\ell)$ whose characteristic polynomials are $X^2+aX+\mu$ and $X^2-aX+\mu$, respectively. These polynomials splits over $\F_\ell$ because $a^2-4\mu=-\mu$ is a square. Then  the matrix  $\begin{pmatrix} 
M_1 & 0 \\ 0 & M_2\end{pmatrix}$ is in $ G\cal{C}^2(\ell)$. 

For $G\cal{C}^3(\ell)$, we take $\mu\in \F_\ell^{\times}$ such that there are $a, b \in \F_\ell^{\times}$ satisfying $a^2=2\mu$ and $b^2=-2\mu$;  for $G\cal{C}^4(\ell)$, we take $\mu\in \F_\ell^{\times}$ such that there are $a, b \in \F_\ell^{\times}$  satisfying  $a^2=4\mu$ and $b=0$;  for $G\cal{C}^5(\ell)$, we take $\mu\in \F_\ell^{\times}$  such that  there are  $a, b \in \F_\ell^{\times}$  satisfying
 $a=0$ and $b=-4\mu$. In each case, find $M_1, M_2\in \GL_2(\F_\ell)$ whose characteristic polynomials are $X^2+aX+\mu$ and $X^2-aX+\mu$, respectively. It is a routine to check these polynomials splits over $\F_\ell$ and the the matrix  $\begin{pmatrix} 
M_1 & 0 \\ 0 & M_2\end{pmatrix}$ is the one we construct.

   To show (2), we  use the list  of  conjugacy classes for $\GSp_4(\F_\ell)$ in \cite[Table 1, pp. 341-346]{Br2015}. From the table, by computing the characteristic polynomial of each conjugacy class, we conclude that if the characteristic polynomial of a matrix $M\in \GSp_4(\F_\ell)$ split into linear polynomials in $\F_\ell[X]$, then $M$ lies in the conjugacy class represented by an element in the Borel subgroup of $\GSp_4(\F_\ell)$. 
   Therefore, $M$ is conjugate to an element in $GB(\ell)$ by  \Cref{gsp-prop} (\ref{group-prop-7}).
   Since the  characteristic polynomial of each element  in $G\cal{C}^i(\ell)$  splits into linear factors,  the element must conjugate to an element in  $GB(\ell)$.

 To show (3), we take an  element $M\in G\cal{C}^i(\ell)$ and any $N\in GU'(\ell)$. Let $\lambda$ be the common diagonal entry of $N$ in $\F_\ell^{\times}$. Consequently, 
    \begin{align}
      \car_{NM}(X)= \prod_{1\leq j\leq 4}(X-\lambda\lambda_j) =X^4 +\mu \lambda^2X^2 + (\mu\lambda^2)^2, \; \quad  i=1,\\
      \car_{NM}(X)= \prod_{1\leq j\leq 4}(X-\lambda\lambda_j)= X^4 - \mu \lambda^2X^2 + (\mu\lambda^2)^2,  \; \quad  i=2,\\
      \car_{NM}(X) = \prod_{1\leq j\leq 4}(X-\lambda\lambda_j)= X^4 + (\mu\lambda^2)^2,  \; \quad  i=3,\\
      \car_{NM}(X)= \prod_{1\leq j\leq 4}(X-\lambda\lambda_j)= (X^2 -\mu \lambda^2)^2,   \;  \quad   i=4, \\
      \car_{NM}(X) = \prod_{1\leq j\leq 4}(X-\lambda\lambda_j)= (X^2 +\mu \lambda^2)^2,  \; \quad  i=5.
    \end{align}
   We observe that in each case,  the inclusion  $GU'(\ell) G\cal{C}^i(\ell)\subseteq G\cal{C}^i(\ell)$ holds. 
   
\end{proof}

We also need to consider the subsets associated to $G\cal{C}^i(\ell)$. 
\begin{align*}
G\cal{C}^i_B(\ell) &:=G\cal{C}^i(\ell)\cap GB(\ell), \  1\leq i\leq 5,\\
\ol{G\cal{C}}^i_B(\ell) & :=\text{image of } G\cal{C}_B^i(\ell) \text{ in } GB(\ell)/GU'(\ell), \  1\leq i\leq 5.
\end{align*}

The  following proposition shows that the cardinality of $\ol{G\cal{C}}_B^i(\ell)$ is bounded independent of $\ell$.
\begin{proposition}\label{gsp-conj-size}
   We have $|\ol{G\cal{C}}_B^i(\ell)| \ll 1$ for $1\leq i\leq 5$.
\end{proposition}
\begin{proof}
For each  $1\leq i\leq 5$ we consider the sets 
    \[
    \ol{G\cal{D}}_B^i(\ell) := \text{image of }
    G\cal{C}_B^i(\ell) \text{ in } GB(\ell)/GU(\ell).
    \]
    Take  $M\in G\cal{C}_B^i(\ell)$. Since $GB(\ell)/GU(\ell)\simeq GT(\ell)$, the matrix $M$ is uniquely determined by its eigenvalues $\lambda_1, \lambda_2, \mu\lambda_1^{-1}, \mu\lambda_2^{-1}\in  \F_\ell^{\times}$. Comparing the coefficients  in each terms of  (\ref{G-conj-1})-(\ref{G-conj-5}), we get the following estimations.
     \begin{align*}
        |\ol{G\cal{D}}_B^1(\ell)| & \leq \sum_{\lambda_1, \lambda_2\in \F_\ell^{\times}} \#\{\mu \in \F_\ell^{\times}: \sum_{1\leq j\leq 2}\lambda_j +\mu \lambda_j^{-1} =0,  
    \lambda_1\lambda_2 +
    \mu \lambda_1\lambda_2^{-1}
    +
    \mu \lambda_2\lambda_1^{-1}
    + 2\mu 
    =\mu\}\\
    & \leq \left(\sum_{\substack{\lambda_1, \lambda_2\in \F_\ell^{\times}\\ \lambda_1^{-1}+\lambda_2^{-1}=0}}1\right)+
   \left( \sum_{\substack{\lambda_1, \lambda_2\in \F_\ell^{\times}\\ \lambda_1^{-1}+\lambda_2^{-1}\neq 0, \; \lambda_1+\lambda_2\neq 0\\
    \lambda_1\lambda_2(\lambda_1^{-1}+\lambda_2^{-1})=(\lambda_1\lambda_2^{-1}+\lambda_1^{-1}\lambda_2+1)(\lambda_1+\lambda_2)}}1\right) \\
    & \leq \left(\sum_{\substack{\lambda_1, \lambda_2\in \F_\ell^{\times}\\ \lambda_1^{-1}+\lambda_2^{-1}=0}}1\right)
    + \left(\sum_{\substack{\lambda_1, \lambda_2\in \F_\ell^{\times}\\ \lambda_1\lambda_2^{-1}+\lambda_1^{-1}\lambda_2+1=1}}1\right)
    \ll \ell.
\end{align*}

Similarly, we have 
\begin{align*}
         |\ol{G\cal{D}}_B^2(\ell)| & \leq \sum_{\lambda_1, \lambda_2\in \F_\ell^{\times}} \#\{\mu \in \F_\ell^{\times}: \sum_{1\leq i\leq 2}\lambda_i +\mu \lambda_i^{-1} =0,  
    \lambda_1\lambda_2 +
    \mu \lambda_1\lambda_2^{-1}+
    \mu \lambda_2\lambda_1^{-1}
    + 2\mu 
    =-\mu\},\\
        & \leq \left(\sum_{\substack{\lambda_1, \lambda_2\in \F_\ell^{\times}\\ \lambda_1^{-1}+\lambda_2^{-1}=0}}1\right)
    + \left(\sum_{\substack{\lambda_1, \lambda_2\in \F_\ell^{\times}\\ 
    \lambda_1\lambda_2^{-1}+\lambda_1^{-1}\lambda_2+3=1}}1\right) \ll \ell,\\
        |\ol{G\cal{D}}_B^3(\ell)| & \leq \sum_{\lambda_1, \lambda_2\in \F_\ell^{\times}} \#\{\mu \in \F_\ell^{\times}: \sum_{1\leq i\leq 2}\lambda_i +\mu \lambda_i^{-1} =0,  
    \lambda_1\lambda_2 +
    \mu \lambda_1\lambda_2^{-1}+
    \mu \lambda_2\lambda_1^{-1}
    + 2\mu 
    =0\}  \\
    & \leq \left(\sum_{\substack{\lambda_1, \lambda_2\in \F_\ell^{\times}\\ \lambda_1^{-1}+\lambda_2^{-1}=0}}1\right)
    + \left(\sum_{\substack{\lambda_1, \lambda_2\in \F_\ell^{\times}\\ \lambda_1\lambda_2^{-1}+\lambda_1^{-1}\lambda_2+2=1}}1\right)
 \ll \ell, \\
   |\ol{G\cal{D}}_B^4(\ell)| & \leq \sum_{\lambda_1, \lambda_2\in \F_\ell^{\times}} \#\{\mu \in \F_\ell^{\times}: \sum_{1\leq i\leq 2}\lambda_i +\mu \lambda_i^{-1} =0,  
    \lambda_1\lambda_2 +
    \mu \lambda_1\lambda_2^{-1}+
    \mu \lambda_2\lambda_1^{-1}
    + 2\mu 
    =-2\mu\}\\
   & \leq \left(\sum_{\substack{\lambda_1, \lambda_2\in \F_\ell^{\times}\\ \lambda_1^{-1}+\lambda_2^{-1}=0}}1\right)
    + \left(\sum_{\substack{\lambda_1, \lambda_2\in \F_\ell^{\times}\\ \lambda_1\lambda_2^{-1}+\lambda_1^{-1}\lambda_2+4=1}}1\right) \ll \ell, 
      \end{align*}
     \begin{align*}
     |\ol{G\cal{D}}_B^5(\ell)| & \leq \sum_{\lambda_1, \lambda_2\in \F_\ell^{\times}} \#\{\mu \in \F_\ell^{\times}: \sum_{1\leq i\leq 2}\lambda_i +\mu \lambda_i^{-1} =0,  
    \lambda_1\lambda_2 +
    \mu \lambda_1\lambda_2^{-1}+
    \mu \lambda_2\lambda_1^{-1}
    + 2\mu 
    =2\mu\}\\
    &  
    \leq \left(\sum_{\substack{\lambda_1, \lambda_2\in \F_\ell^{\times}\\ \lambda_1^{-1}+\lambda_2^{-1}=0}}1\right)
    + \left(\sum_{\substack{\lambda_1, \lambda_2\in \F_\ell^{\times}\\ \lambda_1\lambda_2^{-1}+\lambda_1^{-1}\lambda_2=1}}1\right) \ll \ell.
        \end{align*}
Since the inverse image of $\ol{G\cal{C}}_B^i(\ell) \subseteq GB(\ell)/GU(\ell)$ under the quotient map $GB(\ell)/GU(\ell) \to GB(\ell)/GU'(\ell)$ is exactly $\ol{G\cal{D}}_B^i(\ell)$,  the desired  bounds follow from the fact that
\[
|\ol{G\cal{C}}_B^i(\ell)|=\frac{|\ol{G\cal{D}}_B^i(\ell)|}{|GU'(\ell)/GU(\ell)|}\ll 1, \;  1\leq i\leq 5
\]
\end{proof}

   \subsection{Subgroups and conjugacy classes of \texorpdfstring{$\GL_2(\F_\ell)\times  \GL_2(\F_\ell)$}{Lg}}\label{G-ell-section}
    
    Now we focus on the subsets of
\[
 G(\ell)  := 
    \{(M_1, M_2)\in \GL_2(\F_\ell)\times\GL_2(\F_\ell): \det M_1=\det M_2\}.
\]
We consider the following subgroups of $G(\ell)$. 
\begin{align*}
   B(\ell) & :=
    \{(M_1, M_2)\in G(\ell): M_1  \text{ and } M_2 \text{ are upper triangular}\},\\
    U(\ell) & := 
    \left\{(M_1, M_2)\in G(\ell): M_1=\left(\begin{matrix}
    1 & a_1\\ 0 & 1
    \end{matrix}\right), 
   M_2=\left(\begin{matrix}
    1 & a_2\\ 0 & 1
    \end{matrix}\right), a_1, a_2\in \F_\ell \right\}\\
    U'(\ell) & := 
    \left\{(M_1, M_2)\in G(\ell): M_1=\left(\begin{matrix}
    \lambda & a_1\\ 0 & \lambda
    \end{matrix}\right), 
   M_2=\left(\begin{matrix}
    \lambda  & a_2\\ 0 & \lambda 
    \end{matrix}\right), a_1, a_2\in \F_\ell,\lambda 
 \in\F_\ell^{\times} \right\}, \\
    T(\ell) & := \{(M_1, M_2)\in G(\ell): M_1  \text{ and } M_2 \text{ are diagonal}\},
\end{align*}

We will use the following properties of these groups.
\begin{proposition}\label{group-prop-RM}
Let $\ell\geq 5$. The following statements hold.
    \begin{enumerate}
     \item 
        \begin{align*}
            |G(\ell)| &= (\ell-1)^3\ell^2(\ell+1)^2,\\
            |B(\ell)| & = (\ell-1)^3\ell^2,\\
            |U(\ell)| &= \ell^2,\\
            |U'(\ell)| & = \ell^2(\ell-1),\\
             |T(\ell)| & = (\ell-1)^3.
        \end{align*}
        \label{group-prop-3-RM}
        \item $U(\ell)$ and $U'(\ell)$ are normal subgroups of $B(\ell)$. \label{group-prop-1-RM}
        \item The quotient groups $B(\ell)/U(\ell)$ and $B(\ell)/U'(\ell)$ are abelian. Moreover, we have the isomorphism  $B(\ell)/U(\ell)\simeq T(\ell)$.\label{group-prop-2-RM}
        \end{enumerate} 
\end{proposition}
\begin{proof}
See \cite[Lemma 11, Lemma 12, and Lemma 13 pp. 697-698]{CoWa2023} for the proofs.
\end{proof}

We also need  
the following conjugation invariant sets of $G(\ell)$.
\begin{align}
    \cal{C}^1(\ell) & := 
    \{M\in G(\ell): \car_M(X) = X^4+\mu X^2+\mu^2=\prod_{1\leq j\leq 4}(X-\lambda_j), \lambda_j, \mu\in \F_\ell^{\times}\}, \label{conj-1}\\
    \cal{C}^2(\ell) & := 
    \{M\in G(\ell): \car_M(X) = X^4-\mu X^2+\mu^2=\prod_{1\leq j\leq 4}(X-\lambda_j), \lambda_j, \mu\in \F_\ell^{\times}\}, \label{conj-2}\\
    \cal{C}^3(\ell) & := 
    \{M\in G(\ell): \car_M(X) = X^4  +\mu^2=\prod_{1\leq j\leq 4}(X-\lambda_j), \lambda_j, \mu \in \F_\ell^{\times}\}, \label{conj-3}\\
     \cal{C}^4(\ell) & := 
    \{M\in G(\ell): \car_M(X) = (X^2-\mu)^2=\prod_{1\leq j\leq 4}(X-\lambda_j), \lambda_j, \mu \in \F_\ell^{\times}\},\label{conj-4}\\
         \cal{C}^5(\ell) & := 
    \{M\in G(\ell): \car_M(X) = (X^2+\mu)^2=\prod_{1\leq j\leq 4}(X-\lambda_j), \lambda_j, \mu \in \F_\ell^{\times}\}.\label{conj-5}
\end{align}
Similar to \Cref{conj-class-GSp}, (\ref{conj-1}) and (\ref{conj-2}), (\ref{conj-4}) and (\ref{conj-5}) are actually the same sets, but distinguishing their names will simplify the argument in \Cref{proof-thm-RM}.

\begin{proposition}\label{conj-prop}
Let $\ell$ be an odd prime such that $\kronecker{-1}{\ell}=\kronecker{2}{\ell}=\kronecker{3}{\ell}=1
$. For each $1\leq i\leq 5$, the following statements hold. 
  \begin{enumerate}
    \item $\cal{C}^i(\ell)$ is nonempty and is invariant under conjugation in $G(\ell)$. \label{conj-prop-1-RM}
    \item Each element of $\cal{C}^i(\ell)$ is conjugate over $G(\ell)$ to an element of $B(\ell)$. \label{conj-prop-2-RM}
    \item We have $U'(\ell) \cal{C}^i(\ell)\subseteq \cal{C}^i(\ell)$. \label{conj-prop-3-RM}
    \end{enumerate}
\end{proposition}
\begin{proof}
For (1), the conjugation invariant property  is obvious since the eigenvalues of $G(\ell)$ are invariant under conjugation. 
The proofs of  nonemptiness of $\cal{C}^i(\ell)$ for $1\leq i\leq 5$ is similar to  the proof of  \Cref{gsp-conj-prop} (1).

(2)  follows from \cite[Lemma 15, p. 700]{CoWa2023}, since the characteristic polynomial of each  element of $\cal{C}^i(\ell)$ splits into linear factors.

The proof of (3)  also follows similarly from the proof of  \Cref{gsp-conj-prop} (3).

\end{proof}

Similarly as before, we consider  the sets
\begin{align*}
\cal{C}^i_B(\ell) & :=\cal{C}^i(\ell)\cap B(\ell), \  1\leq i\leq 5,\\
\ol{\cal{C}}^i_B(\ell) & :=\text{image of } \cal{C}_B^i(\ell) \text{ in } B(\ell)/U'(\ell), \  1\leq i\leq 5.
\end{align*}

\begin{proposition}\label{conj-size-RM}
        We have
        $|\ol{\cal{C}}_B^i(\ell)| \ll 1$ for $1\leq i\leq 5$.
 \end{proposition}
\begin{proof}
   First, we consider the sets 
    \[
    \ol{\cal{D}}_B^i(\ell) := \text{image of }
    \cal{C}_B^i(\ell) \text{ in } B(\ell)/U(\ell), 1\leq i\leq 5.
    \]
    Take  $M\in \cal{C}_B^i(\ell)$. Then the image of $M$ in $B(\ell)/U(\ell)\simeq T(\ell)$ is uniquely determined by its eigenvalues. As abelian groups, we have the isomorphism (see  \Cref{gsp-prop} (1), (3) and  \Cref{group-prop-RM} (1), (3))
    \[
    T(\ell)\simeq B(\ell)/U(\ell) \simeq GB(\ell)/GU(\ell)\simeq GT(\ell). 
    \]
 So we can proceed similarly as in the proof of  \Cref{gsp-conj-size}, and derive the bounds 
 \[
     |\ol{\cal{D}}_B^i(\ell)|\ll \ell, \;  1\leq i\leq 5. 
    \]
Finally, we get
\[
|\ol{\cal{C}}_B^i(\ell)|=\frac{|\ol{\cal{D}}_B^i(\ell)|}{|\F_\ell^{\times}|} \ll 1, \;  1\leq i\leq 5. 
\]
\end{proof}

\section{Proof of Theorem 1}
Throughout this section, we keep the notation in \Cref{sec:abelianS} and \Cref{sec:analitic}.
\subsection{The setup}\label{S5.1}

Recalling the classification of possible Weil polynomials in \Cref{ss-Weil-poly}, we obtain the upper bound:
\begin{align}
\pi_{A, ss}(x) 
& \leq \#\{\fp\in \Sigma_K: N(\fp)\leq x, \fp \nmid N_A, \overline{A}_\fp \text{ is supersingular and simple} \} \nonumber \\
& \hspace{10pt}+ \#\{\fp\in \Sigma_K, N(\fp)\leq x, \fp \nmid N_A, \overline{A}_\fp \text{ is supersingular and splits} \} \nonumber \\
&\leq \#\{\fp\in \Sigma_K, N(\fp)=p\leq x, \fp \nmid N_A, P_{A, \fp}(x)= x^4+px^2+p^2\} \label{estimate-1}\\
& \hspace{10pt} + \#\{\fp\in \Sigma_K, N(\fp)=p\leq x, \fp \nmid N_A, P_{A, \fp}(x)= x^4-px^2+p^2\} \label{estimate-2}\\
& \hspace{10pt} +  \#\{\fp\in \Sigma_K, N(\fp)=p\leq x, \fp \nmid N_A, P_{A, \fp}(x)= x^4+p^2\} \label{estimate-3}\\
& \hspace{10pt} + \#\{\fp\in \Sigma_K, N(\fp)=p\leq x, \fp \nmid N_A, P_{A, \fp}(x)= (x^2-p)^2\} \label{estimate-4}\\
& \hspace{10pt} + \#\{\fp\in \Sigma_K, N(\fp)=p\leq x, \fp \nmid N_A, P_{A, \fp}(x)= (x^2+p)^2\} \label{estimate-5}\\
&  \hspace{10pt} +\O(x^{\frac{1}{2}}),
\end{align}
where  (\ref{estimate-1})-(\ref{estimate-4}) together  count prime ideals  $\fp \nmid N_A$ such that $\overline{A}_\fp$ is  supersingular and simple and  (\ref{estimate-5}) counts primes  $\fp \nmid N_A$ such that $\overline{A}_\fp$ is  supersingular and splits. For each $1 \leq i \leq 5$,  we define  $\mathcal{S}_i$  to be the set of primes $\fp$ corresponding, respectively, to   (\ref{estimate-1}), (\ref{estimate-2}), (\ref{estimate-3}),  (\ref{estimate-4}), and (\ref{estimate-5}). 

For each prime $\ell$ and each index $1\leq i\leq 5$, we denote by 
\begin{equation*}
    \tilde{\pi}_{A}^i(\ell, x) := \{\fp\in \cal{S}_i: P_{A, \fp}(X)\pmod \ell \text{ splits into  linear factors in } \F_\ell[X]\}.
\end{equation*}
We first apply \Cref{uncond-lem} to bound $\pi_{A, \mathrm{ss}}(x)$ in terms of $\tilde{\pi}_{A}^i(\ell, x)$, where the prime $\ell$ is chosen from one of the following carefully constructed sets:  
\begin{enumerate}
\item \label{Li-trival} 
We define
\begin{equation*}
      \cal{L}_i :=\begin{cases}
      \left\{\ell \text{ odd}:  \kronecker{-1}{\ell}=\kronecker{3}{\ell}=1 \right\}=\{\ell: \ell \equiv  1\pmod{12}  \}  & i=1, 2 \\
      \left\{\ell \text{ odd}:  \kronecker{-1}{\ell}=\kronecker{2}{\ell}=1 \right\}=\{\ell: \ell \equiv  1\pmod{8}  \}   & i=3 \\
      \{\ell \text{ odd prime}\}    & i=4 \\
       \left\{\ell \text{ odd}:  \kronecker{-1}{\ell}=1 \right\}=\{\ell: \ell \equiv  1\pmod{4}\}    & i=5.
      \end{cases}
    \end{equation*}
 \item \label{Li-RM} If $\End_{\ol{K}}(A)\otimes \Q\simeq \Q(\sqrt{d})$ for some  $d> 0$, 
 we  define for each $i\in \{1, 2, 3, 4, 5\}$
 \begin{align*}
    \cal{L}^{RM}_i & := \cal{L}_i\cap \left\{\ell \text{ odd}:  \ell \text{ is unramified in $K$ and } \kronecker{-d}{\ell}=1\right\} \\
    & = \cal{L}_i\cap \left\{\ell \text{ odd}: \ell \text{ is unramified in $K$ and splits in }\Q(\sqrt{d})\right\},
   \end{align*}
   where $\cal{L}_i$ is defined in (\ref{Li-trival}).
   \item \label{Li-QM}
  If $\End_{\ol{K}}(A)\otimes \Q$ is  isomorphic to a quaternion algebra $D$ with discriminant $d_D$, 
  we  define for each $i\in \{1, 2, 3, 4, 5\}$
   \[
    \cal{L}^{QM}_i := \cal{L}_i\cap \left\{\ell \text{ odd}: \ell>7 \text{ is unramified in $K$ and } \ell\nmid d_D \right\},
   \]   
   where $\cal{L}_i$ is defined in (\ref{Li-trival}).
\end{enumerate} 

 We observe that by the Chebotarev Density Theorem, 
   \[
    \left(\bigcap_{1\leq i\leq 5} \cal{L}_i\right)\bigcap  \{\ell: \ell \text{ split in } \Q(\sqrt{d})\}\bigcap \{\ell: \ell\nmid d_D\} \bigcap \{\ell>7: \ell \text{ unramified in } K\} \neq \emptyset.
   \]
It is worth pointing out that the construction of $\cal{L}_i$ follows a similar rationale to the assumptions of $\ell$  in  \Cref{gsp-conj-prop} and  \Cref{conj-prop}.

    The subsequent lemma provides a criterion for the splitting of   $P_{A, \fp}(X) \pmod \ell$  over $\F_\ell[X]$ using the Legendre symbol $\kronecker{N(\fp)}{\ell}$, where $N(\fp)=p$ is a rational prime.
\begin{lemma}\label{split-lem}
    Let $\ell$ be an odd prime. 
    For each  $\fp\in \cal{S}_i$ and  $\ell\in \cal{L}_i$,  $1\leq i\leq 5$,  the polynomial $P_{A, \fp}(X) \pmod \ell$  splits into linear factors in $\F_\ell[X]$  if and only if
    $
   \kronecker{N(\fp)}{\ell}\neq -1.
    $
\end{lemma}

\begin{proof}
We only present the argument in the case where $i=1$, as the proofs for the other cases are very similar.

   Let $\ell\in \cal{L}_1$, $\fp\in \cal{S}_1$, and write $p=N(\fp)$. If $p=\ell$, then $P_{A, \fp}(X)\pmod \ell$ clearly splits into linear factors. So we assume $p\neq \ell$ and we will find the linear factors  of $P_{A, \fp}(X)\pmod \ell $ in $\F_\ell[X]$. 

We observe that the roots of
\[
P_{A, \fp}(X)\equiv X^4+p X^2+p^2\equiv X^4\left(p^2(1/X)^4+p(1/X)^2+ 1\right) \pmod \ell,
\]
come in pairs $(\alpha_i, p/\alpha_i)\in \ol{\F}_\ell\times \ol{\F}_\ell$ for $i\in \{1, 2\}$. So we can write the following  decomposition 
 \begin{equation}\label{split-mod-ell}
 P_{A, \fp}(X) \equiv (X^2+b_1X+\mu)(X^2+b_2X+\mu) \pmod \ell \quad b_1, b_2, \mu \in \overline{\F}_{\ell}.
  \end{equation} 

First, we show $b_1, b_2, \mu \in \F_{\ell}$ if and only if $\kronecker{p}{\ell}=1$.
  By comparing coefficients on both sides of (\ref{split-mod-ell}),  we get 
  \begin{enumerate}
      \item $\mu^2\equiv p^2 \pmod \ell$;
      \item $b_1= -b_2$;
      \item $b_1 b_2 + 2\mu \equiv p \pmod \ell$.
  \end{enumerate}
  This implies
\begin{equation}\label{beta-reducible}
 \begin{cases}
 \mu \equiv    p \pmod \ell \\
 b := b_1^2=b_2^2\equiv p \pmod \ell
 \end{cases} \ \text{ or } \ 
 \begin{cases}
  \mu \equiv  -p \pmod \ell, \\ b := b_1^2=b_2^2\equiv -3p \pmod \ell.
 \end{cases}
  \end{equation}
   Therefore, the decomposition (\ref{split-mod-ell}) is over $\F_\ell[X]$ if and only if
    \begin{equation}\label{p-3p}
  \kronecker{p}{\ell} = 1 \text{ or } \kronecker{-3p}{\ell}=1, 
  \end{equation}
  which is simply  equivalent to $\kronecker{p}{\ell}=1$ because $\ell\in \mathcal{L}_1$. 
  Suppose $ \kronecker{p}{\ell} = 1$ holds. 
  
  Using the discriminant criteria for the reducibility of quadratic polynomials, we observe that the polynomial (\ref{split-mod-ell}) splits  further into linear polynomials in $\F_\ell[X]$ if and only if 
  \[
\kronecker{b^2-4\mu}{\ell}\neq -1.
\]
In either case of (\ref{beta-reducible}), this is again equivalent to requiring that $\kronecker{p}{\ell} \neq -1$.

Therefore, $P_{A,\fp}(X) \bmod \ell$ splits completely into linear factors in $\F_\ell[X]$ if and only if $\kronecker{p}{\ell} \neq -1$, as claimed.
\end{proof}

Based on \Cref{split-lem}, for each prime $\ell \in \cal{L}_i$ (or in the sets $\cal{L}_i^{\mathrm{RM}}$, $\cal{L}_i^{\mathrm{QM}}$), the counting function $\tilde{\pi}_{A}^i(\ell, x)$ admits the description
\begin{equation}\label{pi-A-ell}
\tilde{\pi}_{A}^i(\ell, x) = \#\{\fp\in \cal{S}_i: \kronecker{N(\fp)}{\ell}\neq -1\}, \; 1\leq i\leq 5.
\end{equation}

\subsection{Abelian surface with $D=\Q$}\label{thm-proof-triv}

We retain the notation and assumptions established in \Cref{S3.1},  \Cref{triv-end-group}, and \Cref{S5.1}. Let $x > 0$ be a large real parameter, and let $t = t(x) \ll (\log x)^{1/2}$ be a positive valued function tending to infinity as $x \to \infty$. For each integer $1\leq i\leq 5$, we  take 
\[
\cal{P}_i:=\{\ell_1< \ldots <\ell_t\}\subseteq \cal{L}_i \cap \left\{\ell \text{ odd}:  \kronecker{-1}{\ell}=\kronecker{2}{\ell}=\kronecker{3}{\ell}=1 \right\}
\]
such that $\ell \ll \frac{\log x}{\log\log x}$ for all $\ell=\ell(x) \in \cal{P}$.  By the Prime Number Theorem in arithmetic progressions, such a set $\cal{P}_i$ exists for sufficiently large $x$. 
Applying  \Cref{uncond-lem}, we derive the inequality
\begin{align}\label{S-i-max-trivial}
   |\cal{S}_i| &
    \leq   \sum_{1\leq j\leq t} \tilde{\pi}_{A}^i(\ell_j, x) + \#\{\fp\in \cal{S}_i: \kronecker{N(\fp)}{\ell}=-1 \text{ for all } \ell\in \cal{P}_i\} \nonumber \\
   & \ll   \sum_{1\leq j\leq t} \tilde{\pi}_{A}^i(\ell_j, x) +\frac{x}{2^t \log(x/P_{i})} \nonumber \\
   &\ll 
t\max_{\ell\in \cal{P}_i}\tilde{\pi}_{A}^i(\ell, x) +\frac{x}{2^t \log(x/P_{i})},
\end{align}
where $\displaystyle P_{i}=\prod_{\ell\in \cal{P}_i}\ell$.

Recall from (\ref{Z-Galois-im}) in  \Cref{S2-Galois-rep} that for for sufficiently large $\ell\in \cal{P}_i$, the Galois group of  $K(A[\ell])/K$ is isomorphic to $\GSp_4(\F_{\ell})$. 
By  \Cref{split-lem}, for each prime $\fp$ counted by $\tilde{\pi}_{A}^i(\ell, x)$, the polynomial $P_{A, \fp}(X)\pmod {\ell}$   splits into linear factors.  Equivalently,   the eigenvalues of   $\ol{\rho}_{A, \ell}(\Frob_\fp)$ are in $\F_\ell^{\times}$. Hence,
\[
\tilde{\pi}_{A}^i(\ell, x) \ll \pi_{G\cal{C}^i(\ell)}(x, K(A[\ell])/K).
\]
Invoking \Cref{gsp-conj-prop}  and applying \Cref{func-ECDT} to the extension $K(A[\ell])/K$ with $G=\GSp_4(\F_\ell)$, $H=GB(\ell)$, $N=GU'(\ell)$, and conjugacy class $\cal{C}=G\cal{C}^i(\ell)$, we obtain
 \begin{align*}
   \tilde{\pi}_{A}^i(\ell, x)  &\ll  \pi_{\overline{G\cal{C}}_B^i(\ell)}(x, K(A[\ell])^{GU'(\ell)}/K(A[\ell])^{GB(\ell)}).
\end{align*}
Applying parts (1) and (3) of  \Cref{gsp-prop},  \Cref{gsp-conj-size},  \Cref{lem-logdK}, and   \Cref{uncond-ECDT}  to the extension $K(A[\ell])^{GU'(\ell)}/K(A[\ell])^{GB(\ell)}$ with the conjugacy-invariant subset $\cal{C}= \overline{G\cal{C}}_B^i(\ell)$  yields
 \begin{equation}\label{pi-tilde}
    \tilde{\pi}_{A}^i(\ell, x)\ll 
  \frac{1}{\ell^2}\frac{x}{\log x},
 \end{equation}
under the restriction that  
 \begin{equation}\label{restriction-ell}
 \log x\gg n_K\ell^4 \log(\ell N_A d_K x).
 \end{equation}
 Combining (\ref{S-i-max-trivial}) and  (\ref{pi-tilde}), we deduce that for all sufficiently large $x$,
  \begin{align*}
    |\cal{S}_i| &\ll  
 \frac{t}{\ell_1^2}\frac{x}{\log x} +\frac{x}{2^t \log(x/\ell_t^t)}
\end{align*}
as long as (\ref{restriction-ell}) is satisfied. 
To optimize this bound, choose $\ell_1=\ell_1(x)$ and $t=t(x)$ such that,  for sufficiently large $x>0$,
\[
 \ell_1\sim c_1\left(\frac{\log x}{n_K\log \log (N_A d_K x)} \right)^{\frac{1}{4}}, \;  t\sim c \log \log x
\]
 for some constants  (sufficiently large) $c > 0$ and (sufficiently small) $c_1 > 0$.
The existence of $t$ primes $\ell_1, \ldots, \ell_t$ in the interval $[\ell_1, 2\ell_1]$ is guaranteed for sufficiently large $x$.

Substituting these choices yields the asymptotic estimate 
\[
  |\cal{S}_i|\ll
  \frac{x (\log \log x)^{\frac{3}{2}}}{(\log x)^{\frac{3}{2}}},
   \]
 where the implicit constant in $\ll$ depends on $A$ and $K$.
 
Summing over $1 \leq i \leq 5$ completes the proof of part (1) of \Cref{main-thm-1}:
\begin{align*}
  \pi_{A, ss}(x) & \ll \sum_{1\leq i\leq 5}|\cal{S}_i|  \ll \frac{x (\log \log x)^{\frac{3}{2}}}{(\log x)^{\frac{3}{2}}}.
\end{align*}

\subsection{Abelian surface with real multiplication}\label{proof-thm-RM}

We keep the notation and assumptions from  \Cref{S3.1},  \Cref{G-ell-section},  and  \Cref{S5.1}. 

We adapt the argument from the previous section to the  primes in the set $\cal{L}_i^{\mathrm{RM}}$. 
Let $x > 0$ be a real parameter tending to infinity, and let $t = t(x) \ll (\log x)^{1/2}$ be a function of $x$ such that $t(x) \to \infty$ as $x \to \infty$. Fix an integer $1 \leq i \leq 5$, and define a set of primes
\[
\cal{P}_i:=\{\ell_1< \ldots< \ell_t\}\subseteq \cal{L}_i^{RM} \cap \left\{\ell \text{ odd}:  \kronecker{-1}{\ell}=\kronecker{2}{\ell}=\kronecker{3}{\ell}=1 \right\}
\]
such that $\ell\ll \frac{\log x}{\log\log x}$ for all $\ell=\ell(x) \in \cal{P}$. 
Applying \Cref{uncond-lem}, we obtain the bound
\begin{align}\label{S-i-max-RM}
   |\cal{S}_i| &
    \leq   \sum_{1\leq j\leq t} \tilde{\pi}_{A}^i(\ell_j, x) + \#\{\fp\in \cal{S}_i: \kronecker{N(\fp)}{\ell}=-1 \text{ for all } \ell\in \cal{P}_i\} \nonumber \\
   &\ll 
t\max_{\ell\in \cal{P}_i}\tilde{\pi}_{A}^i(\ell, x)   +\frac{x}{2^t \log(x/P_{i})},
\end{align}
where $\displaystyle P_{i}=\prod_{\ell\in \cal{P}_i}\ell$. 

We now recall the first part of(\ref{RM-Galois-im}) in \Cref{S2-Galois-rep}, from which it follows that for any sufficiently large prime $\ell \in \cal{P}_i$, the Galois group of the extension $K(A[\ell])/K$ is isomorphic to $G(\ell)$.

Proceeding analogously to the argument in the previous section, we observe that
\[
\tilde{\pi}_{A}^i(\ell, x) \ll \pi_{\cal{C}^i(\ell)}(x, K(A[\ell])/K).
\]
We invoke  \Cref{conj-prop} and apply  \Cref{func-ECDT} with $G=G(\ell)$, $H=B(\ell)$, $N=U'(\ell)$, and $\cal{C}=\cal{C}^i(\ell)$ to get
 \begin{align*}
   \tilde{\pi}_{A}^i(\ell, x)  &\ll  \pi_{\overline{\cal{C}}_B^i(\ell)}(x, K(A[\ell])^{U'(\ell)}/K(A[\ell])^{B(\ell)}), \;  1\leq i\leq 5.
\end{align*}
We further apply parts (1) and (3) of \Cref{group-prop-RM}, together with \Cref{conj-size-RM}, \Cref{lem-logdK}, and \Cref{uncond-ECDT}, to the Galois extension $K(A[\ell])^{U'(\ell)}/K(A[\ell])^{B(\ell)}$ and the conjugation invariant set $\cal{C}= \overline{\cal{C}}_B^i(\ell)$. This yields the estimate
 \begin{equation}\label{pi-tilde-RM}
    \tilde{\pi}_{A}^i(\ell, x)\ll 
  \frac{1}{\ell^2}\frac{x}{\log x},
 \end{equation}
under the restriction that \begin{equation}\label{restriction-ell-RM}
 \log x\gg n_K\ell^2 \log(\ell N_A d_K x).
 \end{equation}
  Therefore, for all sufficiently large $x>0$, we have 
  \begin{align*}
    |\cal{S}_i| &\ll  
 \frac{t}{\ell_1^2}\frac{x}{\log x} +\frac{x}{2^t \log(x/\ell_t^t)}
\end{align*}
as long as (\ref{restriction-ell-RM}) is satisfied for  $\ell\in \cal{P}_i$. 
To satisfy this condition, we choose $\ell_1=\ell_1(x)$ and $t=t(x)$ such that for all sufficiently large $x>0$, 
\[
 \ell_1\sim c_1'\left(\frac{\log x}{n_K\log \log (N_A x)} \right)^{\frac{1}{2}}, \;  t\sim c' \log \log x
\]
for some sufficiently large positive real number $c'$ and some sufficiently small positive real number $c_1'$. The existence of such a sequence of primes follows from standard results on the distribution of primes in short intervals. 
We conclude that for all sufficiently large $x>0$, 
\[
  |\cal{S}_i|\ll
   \frac{x (\log \log x)^{2}}{(\log x)^{2}},
   \]
    where the implicit constant in $\ll$ depends on $A$ and $K$.
    
Summing over all $1 \leq i \leq 5$, we obtain the final estimate
\begin{align*}
  \pi_{A, ss}(x) & \ll \sum_{1\leq i\leq 5}|\cal{S}_i|  \ll \frac{x (\log \log x)^{2}}{(\log x)^{2}},
\end{align*}
 where the  implicit constant in the last $\ll$ depends on $A$ and $K$.

This completes the proof of part (2) of  \Cref{main-thm-1}.

\subsection{Abelian surface with quaternion multiplication}\label{proof-QM}

We keep the notation and assumptions form  \Cref{S3.1},  \Cref{QM-section}, and \Cref{S5.1}. Let $t \geq 1$ be a fixed integer. For each index $ i \in \{4, 5\}$, define the set of primes  
\[
\cal{P}_i:=\{\ell_1< \ldots< \ell_t\}\subseteq \cal{L}_i^{QM} \cap \left\{\ell \text{ odd}:  \kronecker{-1}{\ell}=\kronecker{2}{\ell}=\kronecker{3}{\ell}=1 \right\},
\]
with the additional constraint that each $\ell = \ell(x) \in \cal{P}_i$ satisfies $\ell \ll \frac{\log x}{\log \log x}$.  

From (\ref{QM-Galois-im}) in \Cref{S2-Galois-rep}, we recall that for $\ell$ sufficiently large, the Galois group of the extension $K(A[\ell])/K$ is isomorphic to $\GL_2(\F_\ell)$, embedded diagonally in $\GSp_4(\F_\ell)$. Consequently, for every prime ideal $\fp$ counted by $\pi_{A, \mathrm{ss}}(x)$, the associated Frobenius polynomial $P_{A,\fp}(X)$ is a square in $\Z[X]$. It follows that
\[
\pi_{A, ss}(x)\leq |\cal{S}_4|+|\cal{S}_5|.
\]

For each $i\in \{4, 5\}$, by  \Cref{uncond-lem}, we obtain that  
\begin{align}\label{S-i-max-QM}
   |\cal{S}_i| &
    \leq   \sum_{1\leq j\leq t} \tilde{\pi}_{A}^i(\ell_j, x) + \#\{\fp\in \cal{S}_i: \kronecker{N(\fp)}{\ell}=-1 \text{ for all } \ell\in \cal{P}_i\} \nonumber \\
   &\ll 
t\max_{\ell\in \cal{P}_i}\tilde{\pi}_{A}^i(\ell, x) +\frac{x}{2^t \log(x/P_{i})},
\end{align}
where $\displaystyle P_{i}=\prod_{\ell\in \cal{P}_i}\ell$.
As in the previous analysis, we estimate
\[
\tilde{\pi}_{A}^i(\ell, x) \ll \pi_{\cal{C}'^i(\ell)}(x, K(A[\ell])/K).
\]
We apply \Cref{conj-prop-QM} and \Cref{func-ECDT} to the Galois extension 
$K(A[\ell])/K$ with $G=\GL_2(\F_\ell)$, $H=\cal{B}(\ell)$, $N=\cal{U}'(\ell)$, and $\cal{C}=\cal{C'}^{i}(\ell)$, which yields the estimate
 \begin{align*}
   \tilde{\pi}_{A}^i(\ell, x)  &\ll  \pi_{\ol{\cal{C'}^{i}_{\cal{B}}}(\ell)}(x, K(A[\ell])^{\cal{U}'(\ell)}/K(A[\ell])^{\cal{B}(\ell)}), 
    \; i\in \{4, 5\}.
\end{align*}
Next, we  apply part (1) of \Cref{group-prop-QM},  \Cref{conj-size-QM},  \Cref{lem-logdK}, and  \Cref{uncond-ECDT}
 and  with Galois extension $K(A[\ell])^{\cal{U}'(\ell)}/\Q(A[\ell])^{\cal{B}(\ell)}$, $\cal{C}= \ol{\cal{C'}^{i}_{\cal{B}}}(\ell)$ to get 
 \begin{equation}\label{pi-tilde-QM}
    \tilde{\pi}_{A}^i(\ell, x)\ll 
  \frac{1}{\ell}\frac{x}{\log x},
 \end{equation}
under the restriction that 
\begin{equation}\label{restriction-ell-QM}
 \log x\gg n_K\ell \log(\ell N_A d_K x).
 \end{equation}
 So for all sufficiently large $x>0$, we have 
  \begin{align*}
    |\cal{S}_i| \ll  
 \frac{t}{\ell_1}\frac{x}{\log x} +\frac{x}{2^t \log(x/\ell_t^t)}
\end{align*}
as long as (\ref{restriction-ell-QM}) is satisfied for $\ell\in \cal{P}_i$. 
As before, we choose $\ell_1 :=\ell_1(x)$ and $t :=t(x)$ such that for sufficiently large $x>0$,
\[
 \ell_1\sim c''_1\frac{\log x}{n_K\log \log (N_A d_K x)}, \;  t\sim c'' \log \log x
\]
for some sufficiently large positive real number $c''$ and some sufficiently small positive real number $c''_1$. 

Substituting this estimate into \eqref{S-i-max-QM}, we conclude that for all sufficiently large $x > 0$,
\[
  |\cal{S}_i|\ll 
   \frac{x (\log \log x)^{2}}{(\log x)^{2}},
   \]
    where the  implicit constant in  $\ll$ depends on $A$ and $K$.
    
 Putting all together, we obtain 
\begin{align*}
  \pi_{A, ss}(x) & \ll |\cal{S}_4|+|\cal{S}_5|  \ll \frac{x (\log \log x)^{2}}{(\log x)^{2}},
\end{align*}
 where the  implicit constant in the last $\ll$ depends on $A$ and $K$.
This completes the proof of part (3) of  \Cref{main-thm-1}.

\section{Proofs of \Cref{a2p-split} and   \Cref{a2p-interval}}
We keep the notation and assumptions from   \Cref{a2p-split} and   \Cref{a2p-interval}. In particular,
 $A$ is an abelian surface  defined over a number field $K$ such that   $\End_K(A)\otimes_\Z \Q=\End_{\ol{K}}(A)\otimes \Q= \Q(\sqrt{d})$ for some squrefree integer $d>1$.  Let $N_A$ be the conductor of $A$ and  $\fp$ be a degree 1 prime  of $K$ such that $\fp\nmid N_A$. The Frobenius polynomial of $A$ at $\fp$ is $P_{A, \fp}(X)=X^4+a_{1, \fp}X^3+a_{2, \fp}X^2+pa_{1, \fp}X+p^2\in \Z[X]$.

First, we present the proof of \Cref{a2p-interval} using \Cref{a2p-split}.

\begin{proof}[Proof of \Cref{a2p-interval}]

By elementary considerations, we obtain the following estimate for any $\delta > 0$, 
\begin{align*}\label{cor-3-proof}
  & \#\{\fp\in \Sigma_K:  N(\mathfrak{p}) \leq x, \fp \nmid N_A, \overline{A}_{\mathfrak{p}} \text{ splits}, a_{2, \fp}\in I\} \nonumber \\
   &= \sum_{\substack{t\in \Z\\t \in I}}\#\{\fp\in \Sigma_K:  N(\mathfrak{p}) \leq x, \fp \nmid N_A, \overline{A}_{\mathfrak{p}} \text{ splits}, a_{2, \fp}=t\} \nonumber \\
   &  \ll \#\{t\in \Z: t\in I\} \cdot x^{\frac{1}{2}}\\
   &\ll \frac{x}{(\log x)^{1+\delta}}, \nonumber 
\end{align*} 
 where the final bound follows from part (1) of \Cref{a2p-split} by estimating each summand with $g(p) = t$.  
\end{proof}

\begin{proof}[Proof of  \Cref{a2p-split}]

We begin by deriving basic properties of Frobenius polynomials at primes of split reduction for abelian surfaces with real multiplication. Throughout, we adopt the notation introduced in \Cref{S2.3}.

Let  $\fp\in \Sigma_K$ be a degree 1 prime with $N(\fp)=p$ and assume $p\neq 2$. By  \Cref{RM-poly}, we can write $b_{\fp}(A)=\alpha_\fp+\beta_\fp\sqrt{d}$, where $\alpha_\fp, \beta_\fp\in\Z$, and the Frobenius polynomial factors as 
\begin{align*}
P_{A, \fp}(X) & =X^4+a_{1, \fp}X^3+a_{2, \fp}X^2+pa_{1, \fp}X+p^2\\
& =(X^2+b_{\fp}(A) X+p)(X^2+ (\alpha_\fp+\iota (b_{\fp}(A)) X+p)\\
& =(X^2+(\alpha_\fp+\beta_\fp\sqrt{d}) X+p)(X^2+ (\alpha_\fp-\beta_\fp\sqrt{d})X+p).
\end{align*}
It follows directly that  $a_{1, \fp}=2\alpha_\fp$ and $a_{2, \fp}=2p+\alpha_\fp^2-\beta_\fp^2 d$.

Now suppose that $\ol{A}_\fp$ also splits. Then by  \Cref{main-lem-RM}, there exists an integer $b$ with $|b|\leq 2\sqrt{p}$, such that 
\[
a_{1, \fp}= 2b, \; a_{2, \fp}= 2p+b^2.
\]
Comparing these expressions, we conclude that
\[
\alpha_\fp=b, \; \beta_\fp=0, \; \text{ and } P_{A, \fp}(X) = (X^2+\alpha_\fp X+p)^2. 
\]

(1) Let $g(\cdot)$ be an arithmetic function such that the number of primes $p$ satisfying   $g(p)=2p+m^2$ is uniformly bounded independent of the value of  $m\in \Z$, then 
\begin{align*}
&\#\{\fp\in \Sigma_K: N(\fp)\leq x,  \fp \nmid N_A,  \ol{A}_\fp \text{ splits},  a_{2, \fp}=g(p)\}\\
& = \#\{\fp\in \Sigma_K: N(\fp)=p\leq x,  \fp \nmid N_A,  \ol{A}_\fp \text{ splits},  a_{2, \fp}=g(p)\}+\O(x^{\frac{1}{2}})\\
& \leq  \sum_{\substack{m\in \Z\\|m|\leq 2\sqrt{x}}}\#\{\fp\in \Sigma_K: N(\fp)=p\leq x,  \fp \nmid N_A,  \ol{A}_\fp \text{ splits},  \alpha_{\fp}=m, g(p)=2p+\alpha_\fp^2\} +\O(x^{\frac{1}{2}})\\
& \ll  \sum_{\substack{m\in \Z\\|m|\leq 2\sqrt{x}}}\#\{p\leq x,  g(p)=2p+m^2\} + x^{\frac{1}{2}}\\
& \ll x^{\frac{1}{2}}.
\end{align*}

(2)  By \Cref{lem-7-remark} and the assumption that $A/\Q$ has real multiplication, there exists a  newform  $f$ of weight 2 associated to $A$.  By \cite[Theorem 4]{Ro1997}, for any prime $p\nmid N_A$, 
\[
P_{A, p}(X)=(X^2+a_p(f) X+p)(X^2+ \iota(a_p(f))X+p).
\]
hence the middle coefficient $a_{2, p}$ takes the form $2p+N(a_p(f))$. 
Furthermore, if  $\ol{A}_p$ also splits, then $\alpha_p=a_p(f)\in \Z$. Therefore, 
by  taking   $g(p)=2p+m_0^2$, we get 
\begin{align*}
&\#\{p\leq x: p \nmid N_A,  \ol{A}_p \text{ splits},  a_{2, p}=2p+m_0^2\}\\
& \ll \#\{p\leq x:  p \nmid N_f,  a_p(f)=\pm m_0\}+x^{\frac{1}{2}}.
\end{align*}
The desired unconditional bound now follows from \cite[Theorem 1.4]{ThZa2018}.
\end{proof}

\section{Further discussions}
\subsection{Frobenius trace for \texorpdfstring{$\GL_2$}{Lg}-type abelian surfaces}\label{rm-final}

We have seen that in case (\ref{ss-RM}) of  \Cref{main-thm-1}, the abelian surface $A/\Q$ is  is isogenous over $\Q$ to $A_f$, where $f$ is a non-CM newform of weight 2 and level $N$  \cite{Ri2004}. 
By  \Cref{lem-7-remark} and  \Cref{main-thm-1}  (2), we immediately deduce the following estimate for the number of primes 
 \begin{equation}\label{LT-surface}
 \#\{p\leq x: p\nmid N_{A_f}, a_p(A_f)=0\} \ll \frac{x (\log \log x)^{2}}{(\log x)^{2}},
 \end{equation}
 where $a_p(A_f)$ is the Frobenius trace of $A_f$. This provides an unconditional upper bound toward a special case of the Lang–Trotter Conjecture in the setting of abelian surfaces.  Moreover, due to the modularity of $A_f$ and \Cref{lem-7-remark},  we also obtain the following 
 asymptotic upper bound related to  the Fourier coefficient of $f$:

  \[
 \#\{p\leq x: p\nmid N, a_p(f)=0\} \ll \frac{x (\log \log x)^{2}}{(\log x)^{2}}.
  \]
  Note that this is more general than \cite[Theorem 1.4]{ThZa2018} because  $a_p(f)\in K_f$ is not necessarily an integer.   
 
 We also note that a generalization of the Sato–Tate Conjecture for such abelian surfaces is considered in \cite{Go2014}.
 
Finally, we revisit the conditions (\ref{GL2-type}) and (\ref{GL2-type-2}), and offer another  interpretation in terms of endomorphism algebras.    According to \cite[Theorem 1.2, p. 192]{Py2004}, condition (\ref{GL2-type}) holds if and only if  $\End_{\ol{\Q}}(A)\otimes_\Z \Q$ is an indefinite quaternion algebra over $\Q$, i.e., $A$ has quaternion multiplication). On the other hand, condition (\ref{GL2-type-2}) holds if and only if  $\End_{\ol{\Q}}(A)\otimes_\Z \Q=\End_{\Q}(A)\otimes_\Z \Q$ is a real quadratic field, i.e., $A$ has real multiplication. Furthermore, a modular form $f$ corresponds to  case (\ref{GL2-type}) if and only if it admits a nontrivial inner twist (see \cite{DiRo2005}). 

\subsection{CM abelian varieties}\label{cm-final}

Let $A/K$ be an absolutely simple abelian surface such that  $F: = \End_{\ol{K}}(A)\otimes_{\Z}\Q$ is a quartic CM field. Denote by  $N_A$ the conductor of $A$. In this short section, we discuss he distribution of supersingular primes for the abelian surface $A$. 

We begin with a special case: assume
 $\cal{O}_F\subseteq \End_{\ol{K}}(A)$ and  $F(\subseteq K)$ is a quartic primitive CM field. In this setting, Goren \cite[Theorems 1 and 2]{Go1997} classifies the set of primes  $\fp$ of $K$ such that $\fp \nmid N_A$, $\fp$ is unramified in $F$, and    $\ol{A}_\fp$ is supersingular.  For example, one could obtain that if $K=F$ is a cyclic quartic extension of $\Q$, then the density of rational primes lying below a supersingular prime of  $A/K$ is $\frac{3}{4}$.  In contrast, when  $F$ is a non-cyclic quartic CM field, then the density is at least $\frac{3}{8}$ (see \cite[Theorem 2, case 7]{Go1997}). The density will depend on the decomposition behavior of rational primes in  $F$, its Galois closure $L$ (with $\Gal(L/\Q)\simeq D_4$), and the reflex field of $F$. See also \cite[Example after Theorem 3.1, p. 1260]{Bl2014}.   These results are also compatible with the result that the density of ordinary primes for $A$ is either $\frac{1}{4}$, $\frac{1}{2}$, or 1   \cite[Theorem 2.3, p. 567]{Sa2016}.

More generally, suppose that $A/K$ is an abelian variety of CM type $(F, \Phi)$ (e.g., see \cite[p. 13]{La1983} for the  definition). In this case, Shimura–Taniyama theory \cite{ShTa1961} offers a powerful tool to characterize the supersingular primes of  $A$.   Let $\fp\nmid N_A$ be an unramified prime in $K/\Q$. For any prime $w$ of $F$ such that $w \mid N(\fp)$, we consider the completion  $F_w$ of $F$ at $w$, the image  $H_w$ of $\Hom_{\Q_p}(F_w, \C)$ in $\Hom_\Q(F, \C)$ (under the field isomorphism $\ol{\Q}_p\simeq \C$), and  $\Phi_w = H_w\cap \Phi$. If $\fp$ is a supersingular prime of $A$, then the Shimura-Taniyama formula implies
$
\frac{|\Phi_w|}{|H_w|}=\frac{1}{2}.
$
This relation highlights that computing the density of supersingular primes  for $A$ over $K$ requires detailed knowledge of the arithmetic interaction between $F$ and $K$.   


\subsection{Nonsimple abelian surfaces}
In this section, we examine the distribution of supersingular primes of the (geometrically) non-simple abelian surface $A$ Let  $A/\Q$  be an abelian surface that is $\Q$-isogenous to a product of two non-CM elliptic curves $E/\Q$ and $E'/\Q$. 

By definition, for primes $p\nmid N_A$, the reduction $\ol{A}_p$ is supersingular if and only if both $\ol{E}_p$ and $\ol{E'}_p$ are supersingular.  
If $E$ and $E'$ are $\ol{\Q}$-isogenous, then 
\[
\pi_{A, ss}(x) = \pi_{E, ss}(x)+\O(1) = \pi_{E', ss}(x)+\O(1).
\]
This case has been discussed previously in the introduction. If $E$ and $E'$ are not $\ol{\Q}$-isogenous, the distribution of supersingular primes becomes more complicated. In this setting, Fouvry and Murty \cite{FoMu1995} studied the joint behavior of supersingular reductions and predicted that, for sufficiently large  $x$,
\[
\pi_{A, ss}(x)\sim C_{E, E'} \log\log x,
\]
where  $C
_{E, E'}$ is a constant that only depends on $E$ and $E'$. They also established that this asymptotic holds on average over certain families.

\bibliographystyle{alpha}
\bibliography{References}

\end{document}